\documentclass[a4paper,twoside,11pt]{article}

\usepackage{BeaCedArticle}

\title{Principal minors Pfaffian half-tree theorem}
\author{B\'eatrice de Tili\`ere
\thanks{{\small
Laboratoire de Probabilit\'es et Mod\`eles Al\'eatoires, UMR 7599, Universit\'e
Pierre et Marie Curie, 4 place Jussieu, 
F-75005 Paris.}
{\small\texttt{beatrice.de\_tiliere@upmc.fr.}}
{\small Supported by the ANR Grant 2010-BLAN-0123-02.}
}}
\date{}

\begin{document}

\maketitle

\begin{abstract}
A half-tree is an edge configuration whose superimposition with a perfect
matching is a tree. In this paper, we prove a half-tree theorem for the Pfaffian
principal minors of a skew-symmetric matrix whose column sum is zero;
introducing an explicit algorithm, we fully characterize half-trees
involved. This question naturally arose in the context of
statistical mechanics where we aimed at relating perfect matchings and
trees on the same graph.
As a consequence of the Pfaffian half-tree theorem, we obtain a refined version
of the matrix-tree theorem in the case of skew-symmetric matrices, as well as a
line-bundle version of this result.\\

\end{abstract}

\section{Introduction}

We prove a half-tree theorem for the Pfaffian principal minors of
a skew-symmetric matrix whose column sum is zero. This is a
Pfaffian version of the classical matrix-tree theorem
\cite{Kirchhoff}, see also \cite{Chaiken} and references therein. Introducing
an explicit algorithm, we give a constructive proof of our result
and a full characterization of half-trees
involved. A precise
statement of our main theorem, as well as consequences for the determinant, are
given in Section \ref{sec:11} of the introduction. An outline of the paper is
provided in Section
\ref{sec:13}. Motivations for proving such a result come from statistical
mechanics and are
exposed in Section \ref{sec:12}.

\subsection{Statement of main result}\label{sec:11}

Let $V^R=V\cup R$, where $V=\{1,\dots,n\}$,
$R=\{n+1,\dots,n+r\}$ and $n$ is even.
Let $A^R=(a_{ij})_{\{i,j\,\in V^R\}}$ be a
skew-symmetric matrix whose column sum is zero, \emph{i.e.} satisfying
$
\forall\,i\in V^R,\; \sum_{j\in V^R}a_{ij}=0.
$
Denote by $A=(a_{ij})_{\{i,j\,\in V\}}$ the matrix obtained from $A^R$ by
removing the
$r$ last lines and columns. The matrix $A$ is also skew-symmetric and the
\emph{Pfaffian} of $A$, denoted $\Pf(A)$, is defined as:
\begin{equation*}
\Pf(A)=\frac{1}{2^{\frac{n}{2}}\bigl(\frac{n}{2}\bigr)!} 
\sum_{\sigma\in\Sn}\sgn(\sigma)a_{\sigma(1)\sigma(2)}\dots
a_{\sigma(n-1)\sigma(n)},
\end{equation*}
where $\Sn$ is the set of permutations of $\{1,\dots,n\}$. Using the
skew-symmetry of the matrix $A$ it is possible to avoid summing over all
permutations. Let $\Pn$ be the set of partitions of
$\{1,\dots,n\}$ into $n/2$ unordered pairs, also known as the set
of \emph{pairings}. A permutation $\sigma\in\Sn$ is a \emph{description} of a
pairing $\pi\in\Pn$ if $\{\sigma(1)\sigma(2),\dots,\sigma(n-1)\sigma(n)\}$
represents the pairing. A pairing $\pi$ is described by
$2^{\frac{n}{2}}\bigl(\frac{n}{2}\bigr)!$ permutations: there are
$2^{\frac{n}{2}}$ ways of ordering elements of the pairs and
$\bigl(\frac{n}{2}\bigr)!$ ways of ordering pairs among themselves. Because of
the skew-symmetry of the matrix $A$, the quantity:
$$
\sgn(\sigma)a_{\sigma(1)\sigma(2)}\dots
a_{\sigma(n-1)\sigma(n)},
$$
is independent of the choice of permutation $\sigma$ describing a given
pairing $\pi$. Indeed, choosing another permutation amounts to exchanging
elements of a pair or exchanging pairs. The first operation changes the
sign of the permutation, which is compensated by the change of sign in the
corresponding matrix element. The second operation does not change the
sign, and only changes the order of the matrix elements. As a consequence, the
Pfaffian can be rewritten as:
\begin{align*}
\Pf(A)=\sum_{\pi\in\Pn}\sgn(\sigma_\pi)a_{\sigma_\pi(1)\sigma_\pi(2)}\dots
a_{\sigma_\pi(n-1)\sigma_\pi(n)},
\end{align*} 
where $\sigma_\pi$ is any of the $2^{\frac{n}{2}}\bigl(\frac{n}{2}\bigr)!$
permutations describing the pairing $\pi$. If $n$ is odd, then by convention
$\Pf(A)=0$.

To the matrix $A^R$ one associates the graph $G^R=(V^R,E^R)$, 
where $E^R=\{ij:\,i,j\in V^R,\,a_{ij}\neq
0\}$. Every oriented edge
$(i,j)$ of $G^R$ is assigned a weight $a_{ij}$, thus defining a
skew-symmetric weight function on oriented edges. The matrix $A^R$ is
the \emph{weighted adjacency matrix} of the graph $G^R$. 

A \emph{spanning forest} of $G^R$
is an oriented edge configuration of $G^R$, spanning vertices of
~$V$, such that each connected component is a tree containing
exactly one vertex of
$R$. This vertex is taken to be the root and edges of the component are oriented
towards it. Equivalently, a spanning forest of $G^R$ is an oriented edge
configuration containing no cycle, such 
that each vertex of $V$ has exactly one outgoing edge of
the configuration. A \emph{leaf} of a spanning forest is a vertex with no
incoming edge.

Let $G=(V,E)$ be the graph naturally associated to the matrix $A$.
A \emph{perfect
matching} $M_0$ of $G$ is a subset of edges such that each vertex of $V$ is
incident
to exactly one edge of $M_0$. Note that a perfect
matching of $G$ contains exactly $|V|/2$ edges. In the whole of this paper, we
suppose that $G$ has at least one perfect matching; if this is not the case,
then $\Pf(A)=0$ (see also Section \ref{sec:sec21}). We let $\M$ denote the
set of perfect matchings. 

Let $F$ be a spanning forest of $G^R$, then $F$ is said to be
\emph{compatible} with $M_0$ if it consists of the $|V|/2$ edges of $M_0$ and
of $|V|/2$ edges of $E^R\setminus M_0$. The oriented edge configuration
$F\setminus M_0$ is referred to as a \emph{half-spanning forest}. In the
specific case
where $R$ is reduced to a point, then $F$ is a tree and $F\setminus M_0$ is
referred to as a \emph{half-tree}.
 
\textsc{Example}. Let $V^R=\{1,2,3,4,5\}$, $V=\{1,2,3,4\}$, $R=\{5\}$. Consider
the graphs $G^R$ and $G$ pictured in Figure \ref{fig:fig6}
below. A choice of perfect matching $M_0$ of $G$ is pictured in white, and
$F_1,F_2,F_3$
are examples of spanning trees of $G^R$ compatible with~$M_0$.

\begin{figure}[ht]
\begin{center}
\includegraphics[height=2.5cm]{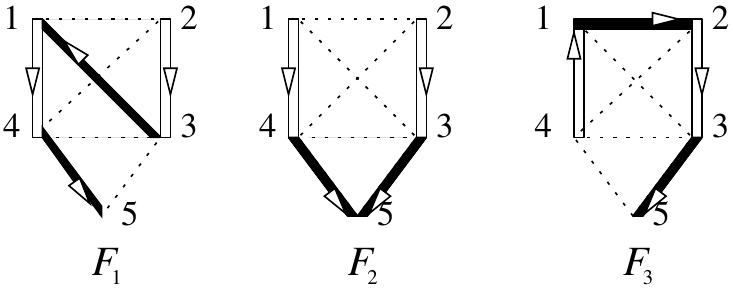} 
\caption{Spanning trees compatible with $M_0$.}\label{fig:fig6}
\end{center}
\end{figure}

Here is the statement of our main theorem.
\begin{thm}[Principal minors Pfaffian half-tree theorem]\label{thm:main}
Let $A^R$ be a skew-symmetric matrix of size $(n+r)\times (n+r)$, whose column
sum is zero, such that $n$ is even; and let $A$ be the matrix obtained from
$A^R$ by removing the $r$ last lines and columns. Let $G^R$ and $G$ be the
graphs naturally constructed from the matrices $A^R$ and $A$, respectively.

For every perfect matching $M_0$ of $G$, the Pfaffian of $A$ is equal to:
$$
\Pf(A)=\sum_{F\in \F(M_0)}\sgn(\sigma_{M_0(F\setminus M_0)})\prod_{e\in
F\setminus M_0}a_{e},
$$ 
where $a_e$ is the coefficient of the matrix $A^R$ corresponding to the
oriented edge $e$; 
$\sgn(\sigma_{M_0(F\setminus M_0)})$ is the sign of the
permutation $\sigma_{M_0(F\setminus M_0)}$ of Definition \ref{def:def2} below;
$\F(M_0)$ is the set of
spanning forests of $G^R$
compatible with $M_0$, satisfying Condition \emph{(C)} of Definition
\ref{def:def0} below.
\end{thm}

\begin{defi}\label{def:def2}
Let $F$ be a spanning forest of $G^R$ compatible with $M_0$. The orientation of
$F$
induces an orientation of edges of the perfect matching $M_0$, and we let
$(i_1,i_2),\dots,(i_{n-1},i_n)$ be a description of the oriented matching.
Then, $\sigma_{M_0(F\setminus M_0)}$ is the permutation:
$$
\sigma_{M_0(F\setminus M_0)}=
\begin{pmatrix}
1&2&\dots&n\\ 
i_1&i_2&\dots&i_n
\end{pmatrix}.
$$
Note that the interchange of two pairs does not change the sign of the
permutation.
\end{defi}

Here is the algorithm used to characterize half-spanning forests which
contribute to $\Pf(A)$.

\textbf{Trimming algorithm}

\underline{Input}: a spanning forest $F$ of $G^R$ compatible with $M_0$.

\underline{Initialization}: $F_1=F$.

\underline{Step $i$, $i\geq 1$}

Since vertices of the graph $G^R$ are labeled by $\{1,\dots,n+r\}$,
there is a natural order on vertices of any subset of $V^R$. We let
${\ell}_1^i$ be the leaf of
$F_i$ with the largest label, and consider the connected component containing
${\ell}_1^i$.
 Start from ${\ell}_1^i$ along the unique path joining ${\ell}_1^i$ to the root
of this component, until the first time one of the following vertices is
reached:
\begin{itemize}
\item[-] the root vertex,
\item[-] a fork, that is a vertex with more than one
incoming
edge,
\item[-] a vertex which is smaller than the leaf ${\ell}_1^i$.
\end{itemize}
This yields a loopless path $\lambda_{{\ell}_1^i}$ starting from ${\ell}_1^i$,
of length
$\geq 1$. Let $F_{i+1}=F_i\setminus \lambda_{{\ell}_1^i}$. If $F_{i+1}$ is
empty,
stop; else go to Step $i+1$.

\underline{End}: since edges are removed at every step, and since $F$ contains a
finite
number of edges, the trimming algorithm ends in finite time $N$.

\begin{defi}\label{def:def0}
A spanning forest $F$ compatible with $M_0$ is said to satisfy
\emph{Condition}~(C) if each of the paths
$\lambda_{{\ell}_1^1},\dots,\lambda_{{\ell}_1^N}$ obtained
from the trimming algorithm, starts from an edge of $M_0$ and has even
length. We let $\F(M_0)$ be the set of spanning forests compatible with $M_0$,
satisfying
Condition (C). 
\end{defi}

\textsc{Example}. Applying the trimming algorithm to each of the spanning
forests
$F_1,F_2,F_3$ of
Figure~\ref{fig:fig6} yields:
\begin{align*}
F_1:&\quad\text{Step 1: }{\ell}_1^1=2,\lambda_{2}=2,3,1.
\quad\text{Step 2: }{\ell}_1^2=1,\lambda_1=1,4,5.\\
F_2:&\quad\text{Step 1: }{\ell}_1^1=2,\lambda_2=2,3,5.
\quad\text{Step 2: }{\ell}_1^2=1,\lambda_1=1,4,5.\\
F_3:&\quad\text{Step 1: }{\ell}_1^1=4,\lambda_4=4,1.
\quad\quad\text{Step 2: }{\ell}_1^2=1,\lambda_1=1,2,3,5.
\end{align*}
The spanning trees $F_1$ and $F_2$ satisfy Condition (C) but not
$F_3$.\\

\begin{rem}$\,$
\begin{itemize}
\item It is interesting to note that taking different perfect matchings $M_0$
yields different families of half-spanning forests. It is not clear a priori,
without using the Pfaffian half-tree theorem, that these families should have
the same total weight.
\item Suppose that we change the labeling of the vertices. Let
$\tilde{A^R}$ be the corresponding skew-symmetric adjacency matrix, and
$\tilde{A}$ be the matrix obtained by removing the $r$ last lines and columns.
As long as the re-labeling does not affect vertices of $R$, the matrix
$\tilde{A}$ is obtained from the matrix $A$ by exchanging lines and columns, so
that $\Pf(\tilde{A})=\pm\Pf(A)$. Applying the Pfaffian half-tree theorem to
the matrices $A$ and $\tilde{A}$ nevertheless yields a different
set of half-spanning forests and again, it is
not clear a priori that they should have the same total weight in absolute
value. Note that taking other principal minors amounts to changing the labeling
of the vertices.
\item In the paper \cite{MasbaumVaintrob}, Masbaum and Vaintrob assign to a
weighted 3-uniform hypergraph a skew-symmetric matrix whose column sum
is zero, and prove that the Pfaffian of any principal minor of this matrix
enumerates signed spanning trees of the 3-uniform hypergraph. The matrix
considered by Masbaum and Vaintrob satisfies the assumptions of Theorem
\ref{thm:main}, implying that the Pfaffian half-tree theorem can also be used.
This naturally raises
the question of possible connections between spanning trees of 3-graphs and
half-spanning trees of Theorem \ref{thm:main}. A detailed account of this
question, illustrated by examples, is provided in Appendix \ref{App:AppendixA}.
Our conclusion is that both theorems can be seen as related to half-spanning
trees, but the latter are of a very different nature. The Pfaffian half-tree
theorem takes its full meaning for (regular) graphs. It can also
be applied for 3-graphs, but the result obtained in that case is rather
different from the one of Masbaum and Vaintrob, and not naturally connected to
spanning trees of 3-graphs.
\end{itemize}
\end{rem}

Using the fact that the determinant of a skew-symmetric matrix is the square of
the Pfaffian, we obtain the following corollary.

\begin{cor}\label{cor:main}
Let $A^R$ be a skew-symmetric matrix of size $(n+r)\times (n+r)$, whose column
sum is zero, such that $n$ is even;  and let $A$ be the matrix obtained from
$A^R$ by removing the $r$ last lines and columns. Let $G^R$ and $G$ be the
graphs naturally constructed from the matrices $A^R$ and $A$ respectively.
 
The determinant of the matrix $A$ is equal to:
\begin{equation*}
\det(A)=\sum_{M_0\in\M}\sum_{F\in\F(M_0)}\prod_{e\in F}a_e,
\end{equation*}
where $a_e$ is the coefficient of the matrix $A^R$ corresponding to the
oriented edge $e$, and $\F(M_0)$ is the set of spanning forests compatible with
$M_0$, satisfying Condition \emph{(C)}.
\end{cor}

\begin{rem}$\,$
\begin{itemize}
\item The fact that principal minors of a skew-symmetric matrix whose column
sum is zero, count spanning forests is also a consequence of the more
general all-minors
matrix-tree theorem (which holds for any matrix whose column sum is zero). A
combinatorial way of proving this result is 
to use the explicit expansion of
configurations due to Chaiken \cite{Chaiken}. This method is not 
satisfactory in our context, since it does not shed a light on how spanning
forests are
obtained from double perfect matchings, which is what we aim for, see Section
\ref{sec:12}. Indeed, the idea of
Chaiken's proof is to expand terms on the diagonal of the matrix and show that
only spanning forests remain. In the case of skew-symmetric matrices, since
diagonal terms are $0$ this amounts to
`artificially' creating configurations which do not exist. As a result of our
proof, we explicitly construct spanning forests
from
double perfect matchings, and identify a specific family of spanning forests
counted by principal minors. In particular, this implies that in the case of
skew-symmetric matrices, specific cancellations occur within spanning forests of
the general matrix-tree theorem, a fact hard to establish without using
Corollary \ref{cor:main}. 
\item An intrinsic definition of $\cup_{M_0\in\M}\F(M_0)$, not using reference
perfect matchings, is given in Remark \ref{rem:main} of Section \ref{sec:sec32}.
\item A line bundle version of this result, in the spirit of \cite{Forman} and
\cite{KenyonVectorBundle}, is proved in Section~\ref{sec:sec33}, see Corollary
\ref{thm:linebundle}.
\end{itemize}
\end{rem}

\subsection{Outline of the paper}\label{sec:13}

\begin{itemize}
\item \underline{Section \ref{sec:sec2}}. In Section \ref{sec:sec21}, we state
the interpretation of the Pfaffian as counting signed perfect matchings of the
graph $G$. Fixing a reference perfect matching $M_0$, we then introduce an
explicit algorithm, which constructs from the superimposition of $M_0$ and a
generic perfect matching $M$ counted by the Pfaffian, a family of
half RC-spanning
forests whose connected components are trees rooted on vertices of $R$, or on
cycles
of even length $\geq 4$; and whose total weight is equal to the contribution of
$M$ to the Pfaffian. The main tool of the algorithm is the `opening' of doubled
edges procedure, described in Section \ref{sec:sec23}. Step 1 of the algorithm
is exposed in Section \ref{sec:sec24}, and the complete algorithm is the subject of
Section \ref{sec:sec25}. A characterization of configurations obtained is given
in Section \ref{sec:sec26}. 
\item \underline{Section 3}. Section \ref{sec:sec31} consists in
the proof of Theorem
\ref{thm:main}. The idea is to show that the contribution of half RC-spanning
forests constructed above, having connected components rooted on cycles of
length $\geq 4$ cancel, and that only the contribution of
spanning forests (rooted on vertices of $R$) remains. The characterization of
configurations obtained from the algorithm is also simplified in the case of
spanning forests, yielding the trimming algorithm of Section \ref{sec:11} of the
introduction. The proof of Corollary~\ref{cor:main} is the subject of 
Section \ref{sec:sec32}. Finally, in Section \ref{sec:sec33}, Corollary
\ref{thm:linebundle} proves a line bundle version of the matrix-tree theorem for
skew-symmetric matrices of Corollary \ref{cor:main}.
\end{itemize}

\subsection{A question from statistical mechanics}\label{sec:12}

As stated in the introduction, the Pfaffian half-tree theorem \ref{thm:main} is
a Pfaffian version of the classical matrix-tree theorem of Kirchhoff. One of its
interesting features is that half-trees involved satisfy specific conditions
characterized by the trimming algorithm, allowing for a refinement of the
matrix-tree theorem in the case of skew-symmetric matrices. As such, the
Pfaffian half-tree theorem is a standalone result. It nevertheless answers a
question raised when working on the paper~\cite{deTiliereCRSF} in
the field of statistical mechanics. In the paper~\cite{deTiliereCRSF} we prove
an explicit
relation, on the level of configurations, between two models of statistical
mechanics: the dimer model on the Fisher graph corresponding to the low
temperature expansion of the critical Ising model (through
Fisher's correspondence \cite{Fisher}), and spanning forests. The
question raised does not rely on the Fisher graph and
can be rephrased in the following, more general framework.

In the setting of statistical mechanics, a perfect matching of a graph is known
as a \emph{dimer configuration}. Assigning non-negative weights to edges of
the graph naturally defines a weight for each dimer configuration (by taking the
product of the edge-weights present in the configuration) and a probability
measure on all dimer configurations of $G$, thus yielding a statistical
mechanics model. The dimer model on planar graphs has been the subject of
extensive studies in the last 50 years, and of huge progresses in the last 15
years, see \cite{KenyonLectures} for an overview. 

A \emph{double dimer configuration} is the superimposition of two dimer
configurations. It consists of a collection of disjoint cycles covering all
vertices of the graph. This is because, by definition of a dimer configuration,
each vertex is incident to exactly one edge of each of the two dimer
configurations, so that in the superimposition, each vertex has degree exactly
two. Our goal is to explicitly construct spanning forests from double
dimer configurations when the model is critical, and to do so on the same
graph, thus proving an unexpected relation, on the level of configurations,
between two models of statistical mechanics. This relation is unexpected because
configurations of the first model consist of cycles, and those of the second
contain no cycle, so that they appear to be of a very different
nature.

When the graph is planar, dimer configurations are counted by the Pfaffian of
the \emph{Kasteleyn matrix}
\cite{Kasteleyn,TemperleyFisher}, which is a weighted adjacency matrix of an
oriented version of the graph; this matrix is skew-symmetric by construction.
It is a general fact that the Pfaffian of an adjacency matrix counts signed
perfect matchings. Signs of perfect matchings come from coefficients of the
matrix and from the signs of permutations naturally assigned to matchings,
see Section \ref{sec:sec21}. The contribution of
\cite{Kasteleyn,TemperleyFisher} is to prove that the orientation of the graph
can be chosen so that signs cancel, implying that all perfect matchings
appear with the same sign. The square of the Pfaffian of the Kasteleyn matrix,
which is the determinant of the matrix, counts double dimer configurations: when
expanding the product, each term consists of two dimer configurations, their
superimposition is a double dimer configuration. 

In the case of the dimer model corresponding to the
critical Ising model, the column sum of the Kasteleyn matrix is zero (when the
graph is embedded on the torus), a fact related to the model being
\emph{critical}. Let us give a little hint at what criticality is.
The Ising model is a model of ferro-magnetism: a magnet is
represented by a graph, vertices of the graph can take
two possible values $\pm1$, and an external temperature influences the system.
When the temperature is zero, all spins are equal to $+1$ or
$-1$; and when the temperature is very high, the configuration is completely
random. At a specific temperature, referred to as the \emph{critical} one, the
system undergoes a phase transition and has a very interesting and rich
behavior, see \cite{ChelkakSmirnov:ising}. In the dimer
interpretation of the Ising model \cite{Fisher}, being critical is 
related to the fact that a certain polynomial in two complex variables has
zeros on the unit torus \cite{Li:critical,CimasoniDuminil}. This polynomial is
the determinant of a modified weight Kasteleyn matrix, and it has zeros on the
unit torus precisely when the column sum of the matrix is zero. This motivates
our choice of taking column-sum equal to zero. 

Our initial question which was constructing spanning forests from double dimer
configurations when the model is critical thus translates into: given a
Kasteleyn matrix whose column sum is zero, how are spanning forests obtained
from double dimer configurations counted by the determinant of the matrix. It
turned out that the only feature required of the Kasteleyn matrix is that of
being skew-symmetric, the specific orientation of the graph did not play a role,
thus taking us away from the setting of statistical mechanics. The question
thus transformed into: how are spanning forests obtained from the signed
superimpositions of perfect matchings counted by the determinant of a
skew-symmetric matrix whose column sum is zero. We obtained more than what we
expected, since we have a result on the Pfaffian. Theorem \ref{thm:main} proves
that principal minors of the Pfaffian of a general skew-symmetric matrix whose column
sum is zero
count a specific family of half-spanning forests, and half-spanning forests
are explicitly constructed from perfect matchings.
 Corollary \ref{cor:main} proves that principal minors of the determinant of
such a matrix count a family of spanning forests, and the latter are explicitly
constructed from superimposition of perfect matchings. Specifying this
result to the case of planar graphs or graphs embedded on the torus, and
Kasteleyn matrices, answers
our initial question. The main drawback
of our result in the context of statistical mechanics is that, even when the
matrix is Kasteleyn and perfect matchings all have positive weights,
corresponding spanning forests might have negative weights.

To close this section on statistical mechanics, let us also mention the work of 
Temperley \cite{Temperley}, Kenyon, Propp and Wilson \cite{KPW} proving that
spanning trees of planar graphs are in bijection with dimer configurations of a
related bipartite graph. The proof consists in a one-to-one correspondence
between configurations. Although their result involves the same kind of objects,
the two are quite different in spirit. In our case, perfect matchings and trees
live on the same graph, the graph must not be bipartite nor even planar, the
weight function on edges of the graph must not be positive, but the column sum
must be zero.

\section{From matchings to half $RC$-rooted spanning forests}\label{sec:sec2}

Let us recall the setting: $A^R$ is a skew-symmetric matrix of size
$(n+r)\times (n+r)$, whose column sum is zero, such that $n$ is even; and $A$ is
the matrix obtained from
$A^R$ by removing the $r$ last lines and columns; $G^R$ and $G$ are the graphs
naturally constructed from the matrices $A^R$ and $A$ in Section \ref{sec:11}
of the introduction. 

\begin{defi}\label{def:def1}
An \emph{$RC$-rooted spanning forest}, referred to as an RCRSF is
an oriented edge
configuration of $G^R$
spanning vertices of $G$, such that each connected component is, either a tree
rooted on a vertex of $R$, or a tree rooted on a cycle of $G$, which we refer to
as a \emph{unicycle}. Edges of each of the components are oriented towards its
root, and edges of the cycles are oriented in one of the two possible
directions.
\end{defi}

\begin{defi}
Let $M_0$ be a reference perfect matching of $G$. An RCRSF $F$ is said to be
\emph{compatible with $M_0$}, 
if it consists of
the $|V|/2$ edges of $M_0$, and of $|V|/2$ edges of
$E^R\setminus M_0$. Moreover cycles of uni-cycles have even length
$\geq 4$, and alternate between edges of $M_0$ and $F\setminus M_0$. The
oriented edge configuration $F\setminus M_0$ is referred to as a
\emph{half-RCRSF}.
\end{defi}

In Section \ref{sec:sec21}, we give the graphical interpretation of the
Pfaffian of the matrix $A$ as counting signed perfect
matchings of $G$. Let $M$ be a generic perfect matching
counted  by the Pfaffian and $M_0$ be a fixed
reference perfect matching of $G$. In Sections
\ref{sec:sec24} and \ref{sec:sec25}, we introduce an explicit algorithm which
constructs, from the superimposition of $M_0$ and $M$, a family of
half $RC$-rooted spanning forests
compatible with $M_0$, whose total weight is equal to the contribution
of $M$ to the Pfaffian. In Section \ref{sec:sec26}, we characterize
$RC$-spanning forests obtained.
Notations used are given in
Section~\ref{sec:sec22}. The main graphical idea of the algorithm is the subject
of Section \ref{sec:sec23}.

\subsection{Graphical interpretation of the Pfaffian}\label{sec:sec21}

Recall that $\Pn$ denotes the set of pairings of
$\{1,\dots,n\}$, and let $\M$ be the set of perfect matchings of $G$. Observing
that
every perfect
matching of $G$ corresponds to a pairing of $\Pn$, and that pairings of $\Pn$
which do not correspond to perfect matchings of $G$ contribute $0$ to the
Pfaffian, we can rewrite $\Pf(A)$ as:
$$
\Pf(A)=\sum_{M\in \M}\sgn(\sigma_M)a_{\sigma_M(1)\sigma_M(2)}\dots
a_{\sigma_M(n-1)\sigma_M(n)},
$$
where $\sigma_{M}$ is a permutation such that
$\{\sigma_M(1)\sigma_M(2),\dots,\sigma_M(n-1)\sigma_M(n)\}$ is a description of
the perfect matching $M$.

Choosing the permutation $\sigma_M$ amounts to choosing an order for the $n/2$
pairs and an order for the two elements of each of the pairs, meaning that
there are $(\frac{n}{2})! 2^{\frac{n}{2}}$ choices. Exchanging two pairs does
not change
the sign nor the corresponding coefficients of the matrix, whereas changing
two elements of a pair changes the sign of the permutation and the sign of
the corresponding element of the matrix. As a consequence, the global sign is
unchanged, and fixing the sign of the permutation amounts to choosing an orientation of edges
of the perfect matching. 

We now specify the choice of sign of the permutation $\sigma_M$ by choosing
an orientation of edges of $M$. Let
$M_0$ be a fixed reference matching of $G$. The
superimposition of $M_0$ and $M$, denoted by $M_0\cup M$, consists of disjoint
doubled edges (covered by both $M_0$ and $M$) and alternating cycles of even
length $\geq 4$, covering all vertices of $G$. Let us, for the moment, consider
doubled edges as cycles of length 2. The orientation of the superimposition
$M_0\cup M$ is fixed by the following rule: for each cycle, the orientation is
compatible with that of the edge $({\ell}_1,{\ell}_1')$, where ${\ell}_1$ is the
smallest
vertex of the cycle, and ${\ell}_1'$ is its partner in $M_0$. This yields an
orientation of edges of $M$, and thus a choice of $\sigma_M$. This procedure
also gives an orientation of edges of $M_0$, and thus a choice of
$\sigma_{M_0}$. Let us also denote by $M_0\cup M$ the oriented superimposition.

\textsc{Example (Figure \ref{fig:fig8}}). In white is a choice of reference
perfect matching $M_0$ and in black are the three possible matchings
$M_1,M_2,M_3$ of the
graph $G$. Edges of the respective superimpositions are oriented according
to the rule described above.

\begin{figure}[ht]
\begin{center}
\includegraphics[height=2.5cm]{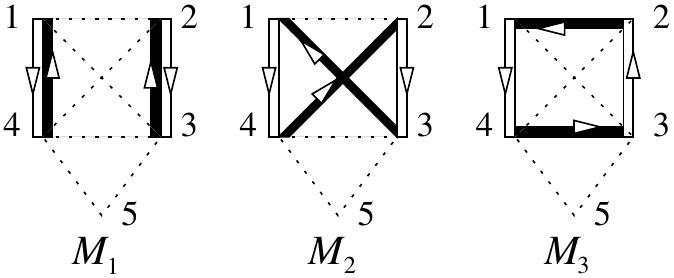}
\caption{Oriented superimposition.\label{fig:fig8}}
\end{center}
\end{figure}

Let $\sigma_{M_0\cup M}$ be the
permutation whose cyclic decomposition corresponds to cycles of the
superimposition $M_0\cup M$. Then,
$$
\sgn(\sigma_{M_0\cup M})=(-1)^{|\D(M_0\cup M)|}(-1)^
{|\C(M_0\cup M)|},
$$
where $\D(M_0\cup M)$ is the set of doubled edges of $M_0\cup M$
and $\C(M_0\cup M)$ is the set of alternating cycles of length $\geq
4$ of $M_0\cup M$. Note that this sign does not depend on the
orientation of the cycles. Following Kasteleyn \cite{Kasteleyn}, the signs of
the permutations $\sigma_{M_0}$ and $\sigma_M$ are related as follows:
\begin{align*}
\sgn(\sigma_M)&=\sgn(\sigma_{M_0})\sgn(\sigma_{M_0\cup M})\\
&=\sgn(\sigma_{M_0})\,(-1)^{|\D(M_0\cup M)|}(-1)^{|\C(M_0\cup M)|}.
\end{align*}
Writing $\sigma_{M_0}$ as $\sigma_{M_0(M)}$ to remember that our
choice of orientation of edges of $M_0$ depends on $M$, the Pfaffian of $A$
can be expressed as:
\begin{equation}\label{equ:pfaffian}
\Pf(A)=\sum_{M\in\M} w_{M_0}(M),
\end{equation}
where
\begin{align*}
w_{M_0}(M)&=\sgn(\sigma_{M_0(M)})(-1)^{|\D(M_0\cup
M)|}(-1)^{|\C(M_0\cup M)|}
\prod_{e\in M}a_{e},
\end{align*}
and $a_e$ is the coefficient of the matrix $A$ corresponding to the oriented
edge $e$ of $M$.

%

\subsection{Notations}\label{sec:sec22}

Let $M_0$ be a fixed reference perfect matching of the graph $G$, $M$ be
a generic perfect matching, and $M_0\cup M$ be the oriented
superimposition of $M_0$ and $M$
constructed in Section \ref{sec:sec21}. In order to shorten notations, we
write $\D$ instead of $\D(M_0\cup M)$ for the set of doubled edges of
the superimposition, $\C$ instead of $\C(M_0\cup M)$ for the set of cycles of
length $\geq 4$, and $w(M)$ instead of $w_{M_0}(M)$.

We now introduce definitions and notations used in the algorithm of Sections
\ref{sec:sec24} and~\ref{sec:sec25}. Let
$V^{\C}$ denote the set of vertices of $V$ which belong to a cycle of $\C$.
For every subset $\D'$ of doubled edges of $\D$, let $V^{\D'}$ denote the set of
vertices of $V$ which belong to doubled edges of $\D'$.

Every vertex $i\in V^{\D}$ belongs to a doubled edge covering vertices
$i$ and $i'$ of $\D$. We denote by $\eb_i$ (or $\eb_i'$) this doubled edge and
define $V_{i}$ to be the set of vertices in the full graph $G^R$, adjacent  to
$i'$ other than $i$.

For every subset $\D'$ of $\D$, denote by
$V_{i}^{\D'}$ the set of vertices of $V_{i}$ which 
belong to doubled edges of $\D'$, and by $(V_i^{\D'})^c$ those which don't. 
Then $V_{i}$ can be
partitioned as: $V_{i}=V_{i}^{\D'}\cup (V_i^{\D'})^c.$

\subsection{Idea of the algorithm}\label{sec:sec23}

The idea of the algorithm is to use the reference configuration $M_0$ as a 
skeleton for opening up doubled edges of the superimposition
~$M_0~\cup~M$. Indeed, because of the condition $\sum_{j\in V^R} a_{ij}=0$, 
configurations of Figure \ref{fig:fig7} have opposite weights.

\begin{figure}[ht]
\begin{center}
\includegraphics[width=10.5cm]{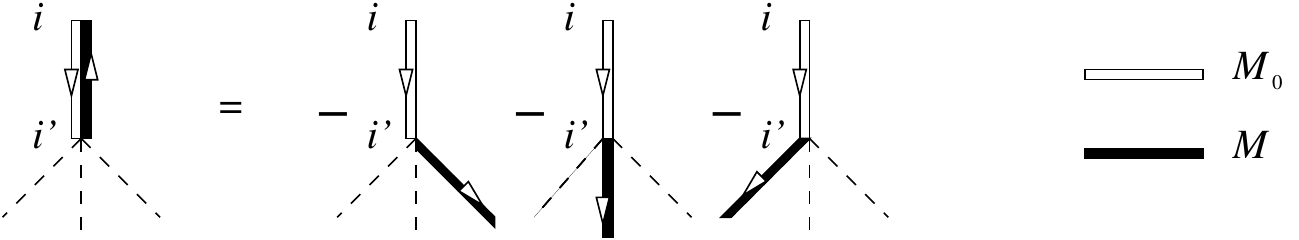}
\caption{`Opening' of doubled edges procedure: $a_{i'i}=-\sum\limits_{j\in
V_i}a_{i'j}$.\label{fig:fig7}}
\end{center}
\end{figure}

There are two main difficulties in realizing this procedure: the first is that
there is, a priori, no natural way of deciding whether to `open' up the
doubled edge at the
vertex $i$ or at the vertex~$i'$. The second is that we want to keep track of
configurations constructed, show that we obtain $RC$-rooted spanning forests,
characterize them and
prove that only spanning forests remain. It turns out that the `opening'
procedure depends strongly on the labeling of vertices.

\subsection{Algorithm: Step 1}\label{sec:sec24}

Recall that the goal of the algorithm is to construct, from the superimposition
$M_0\cup M$ of a reference perfect matching $M_0$ and a generic perfect
matching $M$, a family of half-$RC$-rooted spanning forests of $G^R$ compatible
with $M_0$, whose total weight is equal to the contribution of $M$ to the
Pfaffian. In this section, we introduce the first step of the algorithm,
setting rules for the opening up of doubled edges of $M_0\cup M$. The
complete algorithm, which in essence consists of iterations of Step 1, is the
subject of Section~\ref{sec:sec25}.

\underline{Input}: oriented superimposition $M_0\cup M$.

\underline{Initialization}: if the superimposition $M_0\cup M$ consists of
cycles only, that is if the
set $\D$ is empty, let
${\cal O}_0=\{M\}$ and stop. Else, let ${\cal O}_0=\{\emptyset\}$ and go to the
first
iteration.

\textsc{Example (Figure \ref{fig:fig1}).} Consider Figure \ref{fig:fig1} as
input of the algorithm. The algorithm will be explicitly performed on this
example, throughout Sections \ref{sec:sec24} and \ref{sec:sec25}.

\begin{figure}[ht]
\begin{center}
\includegraphics[height=2cm]{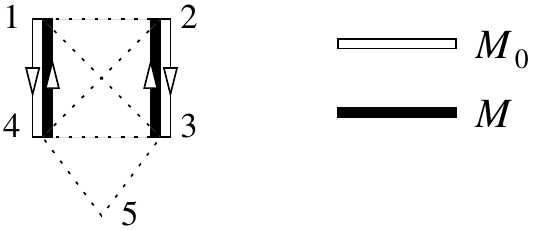} 
\caption{Input $M_0\cup M$ of Step 1.\label{fig:fig1}}
\end{center}
\end{figure}

Since $M_0\cup M$ contains doubled edges, the output ${\cal O}_0$ is
$\{\emptyset\}$.

\underline{Iteration 1}
\begin{itemize}
\item Define ${\ell}_1=\min\{i\in V:\text{ $i$ belongs to a doubled edge of
$\D$}
\}$. Then ${\ell}_1$ is the
partner of a
vertex ${\ell}_1'$ in $M_0$ and $M$. By our choice of orientation for $M_0$ and
$M$, the edge ${\ell}_1{\ell}_1'$ is oriented from ${\ell}_1$ to ${\ell}_1'$ in
$M_0$ and
from ${\ell}_1'$ to ${\ell}_1$ in $M$. For every ${\ell}_2\in V_{{\ell}_1}$,
define:
\begin{align*}
M_{{\ell}_1,{\ell}_2}&=\{M\setminus ({\ell}_1',{\ell}_1)\}\cup
\{({\ell}_1',{\ell}_2)\}\\
w(M_{{\ell}_1,{\ell}_2})&=\sgn(\sigma_{M_0(M)})(-1)^{|\D|-1}(-1)^{|\C|}\prod_{e
\in M_{{\ell}_1,{\ell}_2}}a_{e}.
\end{align*}

\textsc{Example (Figure \ref{fig:fig2})}: ${\ell}_1=1$, $\ell_1'=4$. By
definition, see Section \ref{sec:sec22}, $V_{\ell_1}$ consists of vertices
incident to $\ell_1'=4$ other than $\ell_1=1$, that is,
$V_{{\ell}_1}=V_1=\{2,3,5\}$. This yields configurations $M_{1,2}$,
$M_{1,3}$, $M_{1,5}$.
\begin{figure}[ht]
\begin{center}
\includegraphics[height=2.5cm]{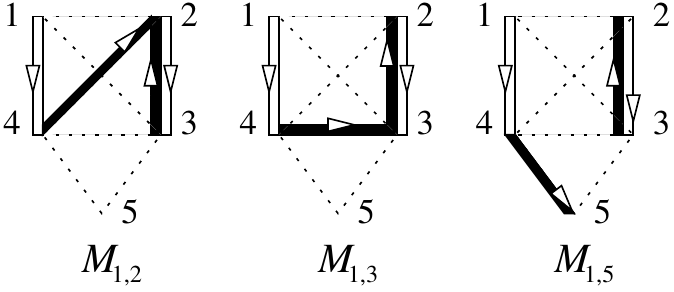} 
\caption{From left to right: black edges are the oriented edge configurations
$M_{1,2}$,
$M_{1,3}$, $M_{1,5}$.\label{fig:fig2}}
\end{center}
\end{figure}

\item Let $D_{{\ell}_1}$ be the set of doubled edges
$\D\setminus\{\eb_{{\ell}_1}\}$.
Then, the set $V_{{\ell}_1}$ can be partitioned as the set of
vertices of $V_{{\ell}_1}$ which belong to a doubled edge of
$\D_{{\ell}_1}$ and the set of those which don't. Using notations of Section
\ref{sec:sec22}, this can be rewritten as:
$V_{{\ell}_1}=V_{{\ell}_1}^{\D_{{\ell}_1}}\cup
(V_{{\ell}_1}^{\D_{{\ell}_1}})^c$.

The output of Iteration 1 is the set of configurations
$M_{{\ell}_1,{\ell}_2}$ such that ${\ell}_2$ does not belong to a doubled edge
of
$\D_{{\ell}_1}$: 
\begin{align*}
{\cal O}_1&=\bigcup_{{\ell}_{2}\in (V_{{\ell}_{1}}^{\D_{{\ell}_1}})^c}
M_{{\ell}_1,{\ell}_2},\\
w({\cal O}_1)&=\sum_{M_{{\ell}_1,{\ell}_2}\in {\cal
O}_1}w(M_{{\ell}_1,{\ell}_2}).
\end{align*}
where by convention, if
$(V_{{\ell}_{1}}^{\D_{{\ell}_1}})^c
=\emptyset$, then ${\cal O}_1=\emptyset$ and $w({\cal O}_1)=0$.

\item The algorithm continues with configurations $M_{{\ell}_1,{\ell}_2}$
where ${\ell}_2$ belongs to a doubled edge of $\D_{{\ell}_1}$. Formally we have:
if
$V_{{\ell}_{1}}^{\D_{{\ell}_1}}=\emptyset$, then stop; else, go to Iteration
$2$.

\textsc{Example}: the set $\D_{{\ell}_1}=\D_{1}$ consists of the doubled edge
$23$.
As a consequence, the set $V_{{\ell}_1}=V_1=\{2,3,5\}$ is partitioned as
$V_1=\{2,3\}\cup\{5\}$, and the output of Iteration $1$ is
${\cal O}_1=\{M_{1,5}\}$. The algorithm continues with $M_{1,2}$ and
$M_{1,3}$.

\end{itemize}

\underline{Iteration $k$, $(k\geq 2)$}

For every ${\ell}_2\in V_{{\ell}_1}^{\D_{{\ell}_1}},\dots, {\ell}_{k}\in
V_{{\ell}_{k-1}}^{\D_{{\ell}_1,\dots,{\ell}_{k-1}}}$, do the following.
\begin{itemize}
\item 
The vertex
${\ell}_k$ is the partner
of a vertex ${\ell}_{k}'$ in $M_0$ and $M$ (since
$\D_{{\ell}_1,\dots,{\ell}_{k-1}}$ is a subset of $\D$).
If ${\ell}_k<{\ell}_k'$, then by our choice of
orientation, the edge ${\ell}_k{\ell}_k'$ is oriented from ${\ell}_k$ to
${\ell}_k'$ in
$M_0$ and from ${\ell}_k'$ to ${\ell}_k$ in $M_{{\ell}_1,\dots,{\ell}_{k}}$. If
${\ell}_k>{\ell}_k'$, then we change the orientation of this edge in $M_0$ and
in
$M_{{\ell}_1,\dots,{\ell}_{k}}$. Let us also denote by
$M_{{\ell}_1,\dots,{\ell}_{k}}$
this new configuration. This change of orientation has the effect of changing
the permutation assigned to $M_0$, and we denote
by $\sigma_{M_0(M_{{\ell}_1,\dots,{\ell}_k})}$ this new permutation. It also
negates
the contribution of $M_{{\ell}_1,\dots,{\ell}_k}$ so that the global
contribution
remains unchanged. For every ${\ell}_{k+1}\in V_{{\ell}_k}$, define:
\begin{align}
M_{{\ell}_1,\dots,{\ell}_{k+1}}&=(M_{{\ell}_1,\dots,{\ell}_k}\setminus
({\ell}_{k}',{\ell}_{k}))\cup
({\ell}_k',{\ell}_{k+1})\nonumber\\
w(M_{{\ell}_1,\dots,{\ell}_{k+1}})&=\sgn(\sigma_{M_0(M_{{\ell}_1,\dots,{\ell}
_k})})(-1)^
{ |\D|-k } (-1)^ { |\C| }
\prod_{e\in M_{{\ell}_1,\dots,{\ell}_{k+1}}}a_{e} \label{equ:weight}
\end{align}

\textsc{Example (Figure \ref{fig:fig3})}. Recall that
$V_{{\ell}_1}^{\D_{{\ell}_1}}=V_1^{\{23\}}=\{2,3\}$, so that
${\ell}_2\in\{2,3\}$. If
${\ell}_2=2$, then $V_{{\ell}_2}=V_2=\{1,4,5\}$, yielding configurations
$M_{1,2,1}$, $M_{1,2,4}$,
$M_{1,2,5}$. If ${\ell}_2=3$, then $V_{{\ell}_2}=V_3=\{1,4\}$, yielding
configurations $M_{1,3,1}$,
$M_{1,3,4}$. 

\begin{figure}[ht]
\begin{center}
\includegraphics[height=5cm]{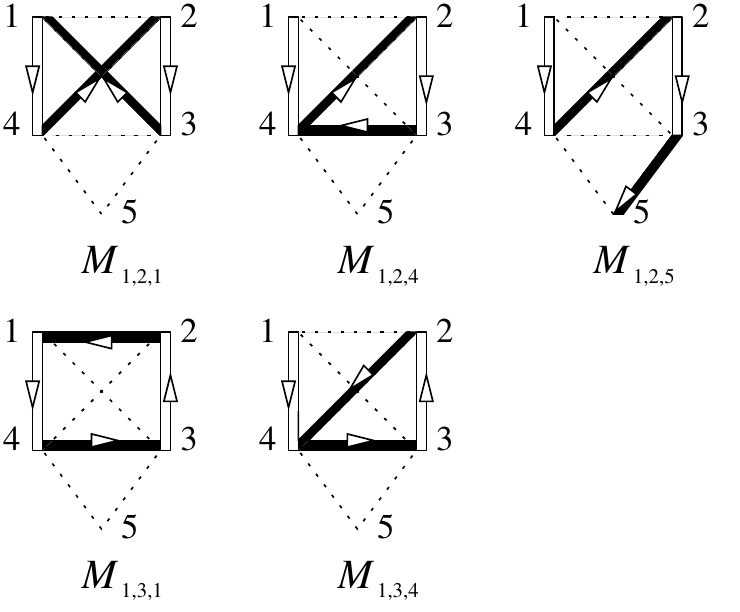} 
\caption{First line, from left
to right, black edges consists of the configurations $M_{1,2,1}$,
$M_{1,2,4}$,
$M_{1,2,5}$. Second line,
from left to right, black edges consists of the configurations
$M_{1,3,1}$,
$M_{1,3,4}$.\label{fig:fig3}}
\end{center}
\end{figure}

\item Let $\D_{{\ell}_1,\dots,{\ell}_k}$ be the set of doubled edges
$\D_{{\ell}_1,\dots,{\ell}_{k-1}}\setminus\{\eb_{{\ell}_k}\}$. Then, the set
$V_{{\ell}_{k}}$
can be partitioned as:
$
V_{{\ell}_{k}}=V_{{\ell}_{k}}^{\D_{{\ell}_1,\dots,{\ell}_k}}
\bigcup (V_{{\ell}_{k}}^{\D_{{\ell}_1,\dots,{\ell}_k}})^c,
$
and the output of Iteration $k$ is the set of configurations
$M_{{\ell}_1,\dots,{\ell}_{k+1}}$ such that ${\ell}_{k+1}$ does not belong to a
doubled edge of $\D_{{\ell}_1,\dots,{\ell}_k}$.

\begin{align*}
{\cal O}_k&=\bigcup_{{\ell}_2\in V_{{\ell}_1}^{\D_{{\ell}_1}}}\dots
\bigcup_{{\ell}_{k}\in V_{{\ell}_{k-1}}^{\D_{{\ell}_1,\dots,{\ell}_{k-1}}}}
\bigcup_{{\ell}_{k+1}\in (V_{{\ell}_{k}}^{\D_{{\ell}_1,\dots,{\ell}_k}})^c}
M_{{\ell}_1,\dots,{\ell}_{k+1}},\\
w({\cal O}_k)&=\sum_{M_{{\ell}_1,\dots,{\ell}_{k+1}}\in
{\cal O}_k}w(M_{{\ell}_1,\dots,{\ell}_{k+1}}).
\end{align*}

\item 
If $V_{{\ell}_{k}}^{\D_{{\ell}_1,\dots,{\ell}_k}}=\emptyset$, then
stop. Else, go to Step $k+1$.

\textsc{Example}: when ${\ell}_2=2$, the set $\D_{1,2}$ is empty
so that $V_{2}$ is partitioned as $\{\emptyset\}\cup\{1,4,5\}$ and the
contribution to the output ${\cal O}_2$ of Iteration 2 is $M_{1,2,1}$,
$M_{1,2,4}$, $M_{1,2,5}$. When ${\ell}_2=3$, the set $\D_{1,3}$
is also empty, implying that $V_{3}$ is partitioned as
$\{\emptyset\}\cup\{1,4\}$ and the contribution to the output ${\cal O}_2$ of
Iteration 2 is $M_{1,3,1}$, $M_{1,3,4}$. After Iteration 2, for every
${\ell}_2\in
V_{{\ell}_1}^{\D_{{\ell}_1}}$,
the set $V_{{\ell}_2}^{\D_{{\ell}_1,{\ell}_2}}$ is empty, so that
the algorithm stops.
\end{itemize}

\underline{End}

Step 1 of the algorithm stops 
at time $m$ for the first time, if it hasn't stopped at time $m-1$, and if
for
every ${\ell}_{2}\in
V_{{\ell}_{1}}^{\D_{{\ell}_1}},\dots,
{\ell}_{m}\in
V_{{\ell}_{m-1}}^{\D_{{\ell}_1,\dots,{\ell}_{m-1}}}$; 
$V_{{\ell}_{m}}^{\D_{{\ell}_1,\dots,{\ell}_m}}=\emptyset$. This
implies in particular that
$(V_{{\ell}_{m}}^{\D_{{\ell}_1,\dots,{\ell}_m}})^c=V_{{\ell}_m}$.
Since the number of doubled edges decreases by $1$ every time an iteration of
the algorithm occurs, and since the number of doubled edges in $\D$ is
finite, we are sure that Step 1 of the algorithm stops in finite time.

\subsubsection{Output of Step 1, geometric properties of
configurations}\label{sec:output}

The output of Step 1 of the algorithm is the set of configurations
$\S_1=\bigcup_{k=0}^{m} {\cal O}_k$. The weight of this set is defined to be
$w(\S_1)=\sum_{k=0}^m w({\cal O}_k)$. 

If the initial superimposition $M_0\cup M$ consists of cycles only, \emph{i.e.}
if the set $\D$ is empty, then $m=0$ and $\S_1=\{M\}$. In all other cases,
the set $\S_1$ can be rewritten in a
more compact way as: 
$$\displaystyle \S_1=
\bigcup_{\gamma_{{\ell}_1}\in\Gamma_{{\ell}_1}}M_{\gamma_{{\ell}_1}},$$ 
where: 
\begin{align*}
\Gamma_{{\ell}_1}=\Big\{&\gamma_{{\ell}_1}:\,\gamma_{{\ell}_1} \text{ is a path
of length $2k$ for some $k\in\{1,\dots,m\}$}:\,
\gamma_{{\ell}_1}={\ell}_1,{\ell}_1',\dots,{\ell}_k,{\ell}_k',{\ell}_{k+1},\\
&{\ell}_1=\min\{i\in V:\,i\text{ belongs to a doubled edge of $\D$}
\}, \\
&\forall j\in\{2,\dots,k\},\,{\ell}_j\in
V_{{\ell}_{j-1}}^{\D_{{\ell}_1,\dots,{\ell}_{j-1}}}
\text{ and }
{\ell}_j'\text{ is the partner of ${\ell}_j$ in $M_0$ and $M$},\\
&{\ell}_{k+1}\in (V_{{\ell}_{k}}^{\D_{{\ell}_1,\dots,{\ell}_k}})^c\Big\}.\\
M_{\gamma_{{\ell}_1}}=\,&M_{{\ell}_1,\dots,{\ell}_{k+1}}.
\end{align*}

Let
$\gamma_{{\ell}_1}={\ell}_1,{\ell}_1',\dots,{\ell}_k,{\ell}_k',{\ell}_{k+1}$ be
a generic path of
$\Gamma_{{\ell}_1}$ for some $k\in\{1,\dots,m\}$, and let
$F_{\gamma_{{\ell}_1}}$
denote the superimposition $M_0\cup M_{\gamma_{{\ell}_1}}$. The
configuration $F_{\gamma_{{\ell}_1}}$ and the path $\gamma_{{\ell}_1^1}$
satisfy the following properties.

$\bullet$ The oriented edge configuration $F_{\gamma_{{\ell}_1}}$:
\begin{enumerate}
\item[(I)] has one outgoing edge at every vertex of $V$, and contains the
path $\gamma_{{\ell}_1}$.
\item[(II)] has $k$ doubled edges less than $M_0\cup M$.
\end{enumerate}
$\bullet$ The oriented path $\gamma_{{\ell}_1}$:
\begin{enumerate}
\item[(III)] has even length $2k$, is alternating (meaning that edges
alternate between $M_0$ and $M_{\gamma_{{\ell}_1}}$). It starts from the
vertex ${\ell}_1$ followed by an edge of $M_0$. 
\item[(IV)] The vertex ${\ell}_1$ is the smallest vertex belonging to a doubled
edge
of $\D$. The $2k$ first vertices of $\gamma_{{\ell}_1}$ are all distinct and the
last vertex
${\ell}_{k+1}$ belongs to
$(V_{{\ell}_k}^{\D_{{\ell}_1,\dots,{\ell}_k}})^c$. Observing that:
$$
(V_{{\ell}_k}^{\D_{{\ell}_1,\dots,{\ell}_k}})^c=
V_{{\ell}_k}\cap(R\cup
V^{\C}\cup\{{\ell}_1,{\ell}_1',\dots,{\ell}_k,{\ell}_k'\}),
$$ we deduce that one of the following holds.
\item[(IV)(1)] If ${\ell}_{k+1}\in R$, then $\gamma_{{\ell}_1}$ is a loopless
oriented
path from ${\ell}_1$ to one of the root vertices of
$R$, and ${\ell}_1$ is a leaf of $F_{\gamma_{{\ell}_1}}$. Since
$R=\{n+1,\dots,n+r\}$,
${\ell}_1$ is smaller than all vertices of
$\gamma_{{\ell}_1}$.
\item[(IV)(2)] If ${\ell}_{k+1}\in V^{\C}$, then $\gamma_{{\ell}_1}$ is a
loopless
oriented path ending at a vertex of one of the cycles of~$\C$ that is, the
connected component containing ${\ell}_1$ is a unicycle with a unique branch.
The vertex ${\ell}_1$ is a leaf of $F_{\gamma_{{\ell}_1}}$ and
is smaller than the $2k$ first vertices of the path, but cannot be a priori
compared to vertices of the cycle of~$\C$. By construction of the orientation of
$M_0\cup M$, see Section \ref{sec:sec21}, the orientation of the
cycle is compatible with that of the edge $(i_1,i_2)$, where $i_1$ is the
smallest vertex of the cycle and $i_2$ is its partner in $M_0$.
\item[(IV)(3)] If ${\ell}_{k+1}\in
\{{\ell}_1,{\ell}_1',\dots,{\ell}_k,{\ell}_k'\}$, then
$\gamma_{{\ell}_1}$ contains a loop of length $\geq 3$. If
${\ell}_{k+1}={\ell}_i$ for some
$i\in\{1,\dots,k\}$, then the
loop has even length and is alternating and the part of $\gamma_{{\ell}_1}$ to
the loop also has even length, is alternating and starts with an edge of $M_0$. 
Moreover, the orientation of the loop is compatible
with the orientation of the edge $({\ell}_i,{\ell}_i')$, and the vertex
${\ell}_1$ is
smaller than all vertices of the path to the cycle and smaller than all vertices
of the cycle.

Note that if ${\ell}_{k+1}\neq {\ell}_1$, then ${\ell}_1$ is a leaf and the
connected component containing ${\ell}_1$ is a unicycle with a unique branch.
Else,
if ${\ell}_{k+1}={\ell}_1$, the connected component is a cycle.

If ${\ell}_{k+1}={\ell}_i'$ for some $i\in\{1,\dots,k\}$, then the loop has odd
length
with two edges of $M$ incident to the vertex ${\ell}_i'$. Observing that the
loop
in both directions is obtained from Step 1 of the algorithm, and using the fact
that the matrix $A$ is skew-symmetric, we
deduce that the contributions of these configurations cancel and we remove them
from the output of Step 1.
Thus we only consider configurations such that ${\ell}_{k+1}={\ell}_i$ for some
$i\in\{1,\dots,k\}$.
\end{enumerate}

\textsc{Example}. The output $\S_1$ of Step 1 is:
$
\S_1=\{M_{1,5},M_{1,2,1},
M_{1,2,4},M_{1,2,5},M_{1,3,1},M_{1,3,4}\}.
$
Configurations $M_{1,5},\,M_{1,2,5}$ are in Case (IV)(1),
configurations $M_{1,2,1},\,M_{1,3,1}$ are in Case (IV)(3) with even cycles
created, and configurations $M_{1,2,4},\,M_{1,3,4}$ are in Case (IV)(3) with odd
cycles created. Contributions of $M_{1,2,4}$ and $M_{1,3,4}$ cancel so that
they are removed from the output. As a consequence the final output of Step 1
is, see also Figure~\ref{fig:fig9}:
$$
\S_1=\{M_{1,5},M_{1,2,1},M_{1,2,5},M_{1,3,1}\},
$$

\begin{figure}[ht]
\begin{center}
\includegraphics[width=8cm]{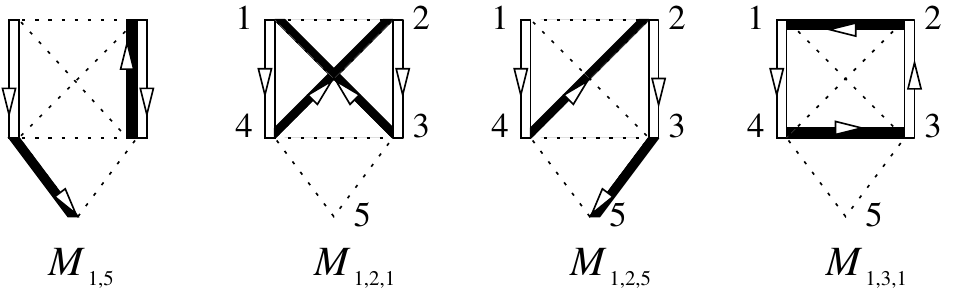} 
\caption{Output of Step $1$ of the algorithm.}\label{fig:fig9}
\end{center}
\end{figure}

\subsubsection{Weight of configurations}\label{sec:weight}

As a consequence of the next two lemmas, we obtain that Step 1 of the algorithm
is weight
preserving \emph{i.e.} $w(\S_1)=w(M)$, see Corollary \ref{cor:1}.

\begin{lem}\label{lem:1} We have:
\begin{itemize}
\item $\displaystyle w(M)=\sum_{{\ell}_2\in
V_{{\ell}_1}}w(M_{{\ell}_1,{\ell}_2})$.\\
\item If $m\geq 2$, then
for every $k\in\{2,\dots,m\}$ and every ${\ell}_2\in
V_{{\ell}_1}^{\D_{{\ell}_1}},\dots,
{\ell}_{k}\in V_{{\ell}_{k-1}}^{\D_{{\ell}_1,\dots,{\ell}_{k-1}}}$:
\begin{equation*}
w(M_{{\ell}_1,\dots,{\ell}_{k}})=\sum_{{\ell}_{k+1}\in
V_{{\ell}_{k}}}w(M_{{\ell}_1,\dots,{\ell}_{k+1}}). 
\end{equation*}
\end{itemize}
\end{lem}
\begin{proof}
Suppose that $m\geq 2$, the proof in the other case being similar. 
For every ${\ell}_{k+1}\in V_{{\ell}_k}$, 
$M_{{\ell}_1,\dots,{\ell}_{k+1}}=\{M_{{\ell}_1,\dots,{\ell}_k}\setminus({\ell}
_k',
{\ell}_k)\}\cup\{({\ell}_k',{\ell}_{k+1})\}$, thus:
\begin{equation*}
\prod_{e\in M_{{\ell}_1,\dots,{\ell}_{k+1}}}a_{e}=
\frac{a_{{\ell}_k',{\ell}_{k+1}}}{a_{{\ell}_k',{\ell}_k}}\prod_{e\in
M_{{\ell}_1,\dots,
{\ell}_k}}a_{e}.
\end{equation*}
By assumption, coefficients of each line of the matrix $A^R$ sum to $0$.
Returning
to the definition of
$V_{{\ell}_k}$, this implies
that $\displaystyle\sum_{{\ell}_{k+1}\in
V_{{\ell}_k}}a_{{\ell}_k',{\ell}_{k+1}}=-a_{{\ell}_k',{\ell}_k}$. Thus,
\begin{equation}\label{equation:1}
\sum_{{\ell}_{k+1}\in
V_{{\ell}_k}}\prod_{e\in M_{{\ell}_1,\dots,{\ell}_{k+1}}}a_{e}
=-\prod_{e\in M_{{\ell}_1,\dots,
{\ell}_k}}a_{e}.
\end{equation}

Combining Equation \eqref{equation:1} with the definition of the weight of
configurations given in Equation \eqref{equ:weight} yields:
\begin{align*}
\sum_{{\ell}_{k+1}\in V_{{\ell}_k}}w(M_{{\ell}_1,\dots,{\ell}_{k+1}}) &=
\sgn(\sigma_{M_0(M_{{\ell}_1,\dots,{\ell}_k})})(-1)^{|\C|}(-1)^{|\D|-k}\sum_{{
\ell}_{k+1}
\in
V_{{\ell}_k}}\prod_{e \in M_ { {\ell}_1 , \dots,{\ell}_{k+1}}}a_{e}\\
&=\sgn(\sigma_{M_0(M_{{\ell}_1,\dots,{\ell}_k})})(-1)^{|\C|}(-1)^{|\D|-(k-1)}
\prod_{e\in M_{{\ell}_1,\dots,{\ell}_k}}a_{e}\\
&=w(M_{{\ell}_1,\dots,{\ell}_k}).
\qedhere
\end{align*}
\end{proof}

\begin{lem}\label{lem:2}
Suppose $m\geq 2$. Then for every $k\in\{2,\dots,m\}$,
$$
\sum_{i=k}^{m}w({\cal O}_i)=\sum_{{\ell}_2\in
V_{{\ell}_1}^{\D_{{\ell}_1}}}\dots 
\sum_{{\ell}_{k}\in
V_{{\ell}_{k-1}}^{\D_{{\ell}_1,\dots,{\ell}_{k-1}}}} w(M_{{\ell}_1,
\dots , {\ell}_ { k} }) .
$$
\end{lem}
\begin{proof}
In order to simplify notations let us write, only in this proof,
$V_{{\ell}_k}^{\D}$ instead of
$V_{{\ell}_k}^{\D_{{\ell}_1,\dots,{\ell}_{k}}}$. Lemma \ref{lem:2} is
proved by backward induction on $k$.

Suppose $k=m$. By definition of the last step of the algorithm,
$V_{{\ell}_{m}}^{\D}=\emptyset$, so that $(V_{{\ell}_m}^{\D})^c=V_{{\ell}_m}$
and:
\begin{align*}
w({\cal O}_m)&=\sum_{{\ell}_2\in V_{{\ell}_1}^\D}\dots 
\sum_{{\ell}_{m}\in V_{{\ell}_{m-1}}^\D}
\sum_{{\ell}_{m+1}\in V_{{\ell}_{m}}}w(M_{{\ell}_1,\dots,{\ell}_{m+1}})\\
&=\sum_{{\ell}_2\in V_{{\ell}_1}^\D}\dots 
\sum_{{\ell}_{m}\in V_{{\ell}_{m-1}}^\D}w(M_{{\ell}_1,\dots,{\ell}_{m}}),\text{
(by
Lemma \ref{lem:1})},
\end{align*}
thus proving the case $k=m$. Suppose that the statement is true for some
$k\in\{3,\dots,m\}$. By Iteration $k-1$ of Step 1 of the algorithm, we know
that:
$$
w({\cal O}_{k-1})=\sum_{{\ell}_2\in V_{{\ell}_1}^\D}\dots 
\sum_{{\ell}_{k-1}\in V_{{\ell}_{k-2}}^\D}
\sum_{{\ell}_{k}\in (V_{{\ell}_{k-1}}^{\D})^c}w(M_{{\ell}_1,\dots,{\ell}_{k}})
$$
Combining this with the induction hypothesis yields:
\begin{align*}
\sum_{i=k-1}^{m}w({\cal O}_i)&=w({\cal O}_{k-1})+\sum_{i=k}^{m}w({\cal O}_i)\\
&=\sum_{{\ell}_2\in V_{{\ell}_1}^\D}\dots 
\sum_{{\ell}_{k-1}\in V_{{\ell}_{k-2}}^\D}
\Bigl(
\sum_{{\ell}_{k}\in(V_{{\ell}_{k-1}}^{\D})^c}+
\sum_{{\ell}_{k}\in V_{{\ell}_{k-1}}^{\D}}\Bigr)
w(M_{{\ell}_1,\dots,{\ell}_{k}})\\
&= \sum_{{\ell}_2\in V_{{\ell}_1}^\D}\dots 
\sum_{{\ell}_{k-1}\in V_{{\ell}_{k-2}}^\D}
\sum_{{\ell}_{k}\in V_{{\ell}_{k-1}}}w(M_{{\ell}_1,\dots,{\ell}_{k}})\\
&=\sum_{{\ell}_2\in V_{{\ell}_1}^\D}\dots 
\sum_{{\ell}_{k-1}\in V_{{\ell}_{k-2}}^\D}w(M_{{\ell}_1,\dots,{\ell}_{k-1}})
\text{ (by
Lemma \ref{lem:1})},
\end{align*}
proving the statement for $k-1$ and ending the proof of Lemma \ref{lem:2}.
\end{proof}

\begin{cor}\label{cor:1}
$$
w(\S_1)=w(M).
$$ 
\end{cor}
\begin{proof}
Suppose $m\geq 1$. Then, 
\begin{align*}
w(\S_1)&=w({\cal O}_1)+\sum_{k=2}^m w({\cal O}_k)\\
 &=w({\cal O}_1)+\sum_{{\ell}_2\in
V_{{\ell}_1}^{\D_{{\ell}_1}}}w(M_{{\ell}_1,{\ell}_2}),\text{ (by Lemma
\ref{lem:2})}\\
&=\sum_{{\ell}_2\in
(V_{{\ell}_1}^{\D_{{\ell}_1}})^c}w(M_{{\ell}_1,{\ell}_2})+\sum_{{\ell}_2\in
V_{{\ell}_1}^{\D_{{\ell}_1}}}
w(M_{{\ell}_1,{\ell}_2}), \text{ (by definition of ${\cal O}_1$)}\\
&=\sum_{{\ell}_2\in V_{{\ell}_1}}w(M_{{\ell}_1,{\ell}_2})\\
&=w(M), \text{ (by Lemma \ref{lem:1})}.
\end{align*}
When $m=0$, $\S_1=\{M\}$, and the conclusion is immediate.
\end{proof}

\subsection{Complete algorithm}\label{sec:sec25}

Let $M_0$ be a reference perfect matching of the graph $G$ and let $M$ be a
generic one. Recall that $\C$ denotes the set of cycles of length $\geq 4$ of
the superimposition $M_0\cup M$, and $\D$ denotes its set of doubled edges.

In Section \ref{sec:sec24}, we established Step 1 of the
algorithm, starting from the superimposition $M_0\cup M$, yielding a set of
oriented edge configurations $\S_1$ through the opening of doubled edges
procedure, whose total weight is equal to the contribution of $M$ to the
Pfaffian. In this section, we introduce the complete algorithm, which in essence
consists of iterations of Step~1 performed until no doubled edges of $\D$
remain.

Let us directly handle the following trivial case. If $M_0\cup M$ consists of
cycles only, that is, if the set $\D$ is the empty, then the opening
of edges procedure does not start, and recall that the output of Step 1 is
$\S_1=\{M\}$. The same holds for the complete algorithm and its
output is $\T=\{M\}$. 

\subsubsection{Step 1 of the complete algorithm}

Assume that the initial superimposition contains at least one doubled edge,
\emph{i.e.} $\D\neq\emptyset$.

Notations are complicated by the fact that the
algorithm depends on the
labeling of the vertices and that iterations of Step 1 depend on previous
steps. We thus need many indices to keep track of everything rigorously, but
one should keep in mind that, in essence, we are iterating Step
1. Let us add sub/superscripts
to Step 1 of Section~\ref{sec:sec24}. That is, ${\ell}_1$ becomes
${\ell}_1^1$, Iteration $k$ becomes
$k_1$ and Step 1 ends at time $m_1$. The set of configurations obtained
from Step 1 is
$
\S_{1}=\bigcup_{\gamma_{{\ell}_1^1}\in
\Gamma_{{\ell}_1^1}}M_{\gamma_{{\ell}_1^1}}$,
and its weight is
$w(\S_1)=\sum_{\gamma_{{\ell}_1^1}\in\Gamma_{{\ell}_1^1}}w(M_{\gamma_{{\ell}_1^1
}})$.

For every $\gamma_{{\ell}_1^1}\in\Gamma_{{\ell}_1^1}$, let
$\D_{\gamma_{{\ell}_1^1}}$ be
the set of doubled edges of the superimposition $M_0~\cup~
M_{\gamma_{{\ell}_1^1}}$. If $\D_{\gamma_{{\ell}_1^1}}=\emptyset$, then stop;
else
go to Step $2$. 

\underline{Output of Step 1 of the complete algorithm}. It is the subset
$\T_1$ of $\S_1$, consisting of configurations $M_{\gamma_{{\ell}_1^1}}$
where $\gamma_{{\ell}_1^1}\in\Gamma_{{\ell}_1^1}$, and
$\D_{\gamma_{{\ell}_1^1}}$ is empty. Formally,
$$
\T_1=\bigcup_{\gamma_{{\ell}_1^1}\in(\widetilde{\Gamma}_{{\ell}_1^1})^c}
M_{\gamma_{{\ell}_1^1}},
$$
where $\widetilde{\Gamma}_{{\ell}_1^1}$ is the set of paths
$\gamma_{{\ell}_1^1}$ of
$\Gamma_{{\ell}_1^1}$ such that $\D_{\gamma_{{\ell}_1^1}}$ is non-empty. If for
all
$\gamma_{{\ell}_1^1}\in \Gamma_{{\ell}_1^1}$, the set $\D_{\gamma_{{\ell}_1^1}}$
is non-empty, then $\T_1=\emptyset$.

\textsc{Example}. Recall that the output of Step 1 is 
$\S_1=\{M_{1,5},M_{1,2,1},M_{1,2,5},M_{1,3,1}\}$. The
set of doubled edges of the superimposition of $M_0$ and
$M_{1,2,1},M_{1,2,5},M_{1,3,1}$ is empty, so that the output of
Step 1 of the complete algorithm is
$\T_1=\{M_{1,2,1},M_{1,2,5},M_{1,3,1}\}$, and the algorithm
continues with the configuration $M_{1,5}$.

\subsubsection{Step $j$ of the complete algorithm, $j\geq 2$}

For every $\gamma_{{\ell}_1^1}\in\widetilde{\Gamma}_{{\ell}_1^1},\dots,
\gamma_{{\ell}_1^{j-1}}\in\widetilde{\Gamma}_{{\ell}_1^{j-1}}(\gamma_{{\ell}_1^1
},\dots,
\gamma_{{\ell}_1^{j-2}})$, perform Step~1
of the algorithm with the initial
superimposition
$M_0\cup M_{\gamma_{{\ell}_1^1},\dots,\gamma_{{\ell}_1^{j-1}}}$. That is,
define ${\ell}_1^j=\min\{i\in V:\,i \text{ belongs to a doubled edge of }
\D_{\gamma_{{\ell}_1^1},\dots,\gamma_{{\ell}_1^{j-1}}}\}$, and iterate until
the
algorithm ends at some time $m_j$. Everything works out in the same way because
$\D_{\gamma_{{\ell}_1^1},\dots,\gamma_{{\ell}_1^{j-1}}}$ is a subset of $\D$.
The
output is the set of oriented edge configurations: 
\begin{equation*}
\S_j(\gamma_{{\ell}_1^1},\dots,\gamma_{{\ell}_1^{j-1}})
=
\bigcup_{\gamma_{{\ell}_1^j}\in
\Gamma_{{\ell}_1^{j}}(\gamma_{{\ell}_1^1},
\dots,\gamma_{{\ell}_1^{j-1}})}M_{\gamma_{{\ell}_1^1},\dots,\gamma_{{\ell}_1^{
j}}},
\end{equation*}
where 
\begin{align}\label{equ:gamma1}
\Gamma_{{\ell}_1^j}(\gamma_{{\ell}_1^1},\dots,\gamma_{{\ell}_1^{j-1}}) 
&=\Bigg\{\gamma_{{\ell}_1^j}:\,\gamma_{{\ell}_1^j}\text{ is a path
}{\ell}_1^j,{\ell}_1^{j'},\dots,{\ell}_{k_j}^j,{\ell}_{k_j}^{j'},{\ell}_{k_j+1}
^j,
\text{ for some $k_j\in\{1,\dots,m_j\}$},\nonumber\\
&\text{such that }
{\ell}_1^j=\min\{i\in V:\,i \text{ belongs to a doubled edge of }
\D_{\gamma_{{\ell}_1^1},\dots,\gamma_{{\ell}_1^{j-1}}}\}\nonumber\\
&\forall i\in\{2,\dots,k_j\},\,{\ell}_i^j\in
V_{{\ell}_{i-1}^j}^{
\D_{\gamma_{{\ell}_1^1},\dots,\gamma_{{\ell}_1^{j-1}},{\ell}_1^j,\dots,{\ell}_
{i-1}^j}
},\nonumber\\
&\text{ and }{\ell}_i^{j'}\text{ is the
partner of ${\ell}_i^j$ in $M_0$ and $M$},\nonumber\\
&
{\ell}_{k_j+1}^j\in 
\Bigl(V_{{\ell}_{k_j}^j}^{
\D_{\gamma_{{\ell}_1^1},\dots,\gamma_{{\ell}_1^{j-1}},{\ell}_1^j,\dots,{\ell}_
{k_j}^j}}
\Bigr)^c
\Bigg\}.\\
M_{\gamma_{{\ell}_1^1},\dots,\gamma_{{\ell}_1^{j}}}
=\,&M_{\gamma_{{\ell}_1^1},\dots,\gamma_{{\ell}_1^{j-1}},{\ell}_1^j,
\dots,{\ell}_{k_j+1}^j},\nonumber\\
w(M_{\gamma_{{\ell}_1^1},\dots,\gamma_{{\ell}_1^{j}}})&=
\sgn(\sigma_{M_0(M_{\gamma_{{\ell}_1^1},\dots,\gamma_{{\ell}_1^{j-1}},{\ell}
_1^j,\dots,{\ell}_
{ k_j}^j})
})(-1)^{|\D|-(k_1+\dots+k_j)}(-1)^{|\C|}
\prod_{e\in
M_{\gamma_{{\ell}_1^1},\dots,\gamma_{{\ell}_1^{j}}}}a_{e}.\nonumber
\end{align}

For every $\gamma_{{\ell}_1^j}\in
\Gamma_{{\ell}_1^j}(\gamma_{{\ell}_1^1},\dots,\gamma_{{\ell}_1^{j-1}})$, do the
following:
if $\D_{\gamma_{{\ell}_1^1},\dots,\gamma_{{\ell}_1^j}}$ is empty, stop; else go
to Step $j+1$.

\underline{Output of Step $j$ of the complete algorithm}.

Let $\T_j(\gamma_{{\ell}_1^1},\dots,\gamma_{{\ell}_1^{j-1}})$ be the subset of
$\S_j(\gamma_{{\ell}_1^1},\dots,\gamma_{{\ell}_1^{j-1}})$, consisting of
configurations $M_{\gamma_{{\ell}_1^1},\dots,\gamma_{{\ell}_1^j}}$ such that
$\D_{\gamma_{{\ell}_1^1},\dots,\gamma_{{\ell}_1^j}}$ is empty. Formally,
$$
\T_j(\gamma_{{\ell}_1^1},\dots,\gamma_{{\ell}_1^{j-1}})=
\bigcup_{
\gamma_{{\ell}_1^j}\in(\widetilde{\Gamma}_{{\ell}_1^j}(\gamma_{{\ell}_1^1},
\dots,
\gamma_{{\ell}_1^{j-1}}))^c
}M_{\gamma_{{\ell}_1^1},\dots,\gamma_{{\ell}_1^j}},$$
where
$\widetilde{\Gamma}_{{\ell}_1^j}(\gamma_{{\ell}_1^1},\dots,
\gamma_{{\ell}_1^{j-1}})$ is the subset of paths $\gamma_{{\ell}_1^j}$ of
$\Gamma_{{\ell}_1^j}(\gamma_{{\ell}_1^1},\dots,
\gamma_{{\ell}_1^{j-1}})$ such that
$\D_{\gamma_{{\ell}_1^1},\dots,\gamma_{{\ell}_1^j}}$
is non-empty. If $\D_{\gamma_{{\ell}_1^1},\dots,\gamma_{{\ell}_1^j}}$ is non-
empty,
then $\T_j(\gamma_{{\ell}_1^1},\dots,\gamma_{{\ell}_1^{j-1}})=\emptyset$.

Then, the output $\T_j$ of Step $j$ of the complete algorithm is:
\begin{equation*}
\T_j=\bigcup_{\gamma_{{\ell}_1^1}\in \widetilde{\Gamma}_{{\ell}_1^1}}
\dots
\bigcup_{\gamma_{{\ell}_1^{j-1}}\in\widetilde{\Gamma}_{{\ell}_1^{j-1}}(\gamma_{{
\ell}_1^1},
\dots,
\gamma_{{\ell}_1^{j-2}} )}
\bigcup_{\gamma_{{\ell}_1^j}\in
(\widetilde{\Gamma}_{{\ell}_1^j}(\gamma_{{\ell}_1^1},\dots,\gamma_{{\ell}_1^{
j-1}}))^c}
M_{\gamma_{{\ell}_1^1},\dots,\gamma_{{\ell}_1^{j}}}.
\end{equation*}
For convenience, we shall
also use the notation $\Gamma_j$ for the
paths $(\gamma_{{\ell}_1^1},\dots,\gamma_{{\ell}_1^j})$ involved in $\T_j$,
\emph{i.e}:
$$
\T_j=
\bigcup_{(\gamma_{{\ell}_1^1},\dots,\gamma_{{\ell}_1^j})\in\Gamma_j}M_{\gamma_{
{\ell}_1^1 } ,\dots,\gamma_{{\ell}_1^j}}.
$$
The weight of $\T_j$ is the sum of the weights of the
configurations it contains.\\

\textsc{Example}. In Step 2, we perform Step $1$ of the algorithm starting from
the initial superimposition $M_0\cup M_{1,5}$. The
latter contains one doubled edge $23$, thus the vertex ${\ell}_1^2$ is the
smallest
of 2 and 3, that is 2. The output $\S_2$ consists of the configurations
$M_{1,5;2,1}, M_{1,5;2,4},M_{1,5;2,5}$, 
depicted in Figure \ref{fig:fig4} below. 

\begin{figure}[ht]
\begin{center}
\includegraphics[height=2.5cm]{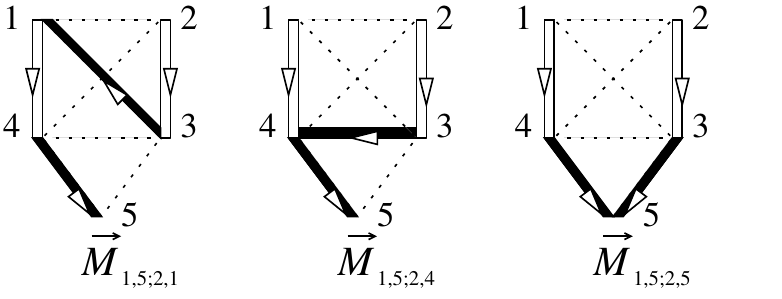} 
\caption{Output of Step $2$ of the algorithm.}\label{fig:fig4}
\end{center}
\end{figure}

The superimposition of $M_0$ and the above three configurations contains
no doubled edges. As a consequence, the complete algorithm stops and the output
$\T_2$ of Step $2$ is:
$$\T_2=\{M_{1,5;2,1}, M_{1,5;2,4},M_{1,5;2,5}\}.$$

\subsubsection{End and output of the complete algorithm}

The algorithm stops at Step $T$ for the first time, if it hasn't stopped at
time $T-1$, and if for every
$\gamma_{{\ell}_1^1}\in\widetilde{\Gamma}_{{\ell}_1^1},\dots,
\gamma_{{\ell}_1^{T-1}}\in\widetilde{\Gamma}_{{\ell}_1^{T-1}}(\gamma_{{\ell}_1^1
},\dots,
\gamma_ { {\ell}_1^ { T-2 } } )$ ,
$\gamma_{{\ell}_1^T}\in\Gamma_{{\ell}_1^{T}}(\gamma_{{\ell}_1^1},\dots,\gamma_
{{\ell}_1^{T-1}})$, the superimposition
$M_0\cup M_{\gamma_{{\ell}_1^1},\dots,\gamma_{{\ell}_1^T}}$ contains no
doubled edge. This implies in particular that 
$(\widetilde{\Gamma}_{{\ell}_1^{T}}(\gamma_{{\ell}_1^1},\dots,\gamma_
{{\ell}_1^{T-1}}))^c=\Gamma_{{\ell}_1^{T}}(\gamma_{{\ell}_1^1},\dots,\gamma_
{{\ell}_1^{T-1}})$. Since the number of doubled edges decreases at every step
and
since $\D$ is finite, we are sure that this happens in finite time.

The output $\T$ of the complete algorithm is :
$$
\T=\bigcup\limits_{j=1}^T \T_j.
$$


\textsc{Example}. The output of the complete algorithm is:
$$
\T=\T_1\cup\T_2=\{M_{1,2,1},M_{1,2,5},M_{1,3,1},
M_{1,5;2,1}, M_{1,5;2,4},M_{1,5;2,5}\},$$ summarized in Figure
\ref{fig:fig5} below.

\begin{figure}[ht]
\begin{center}
\includegraphics[width=11cm]{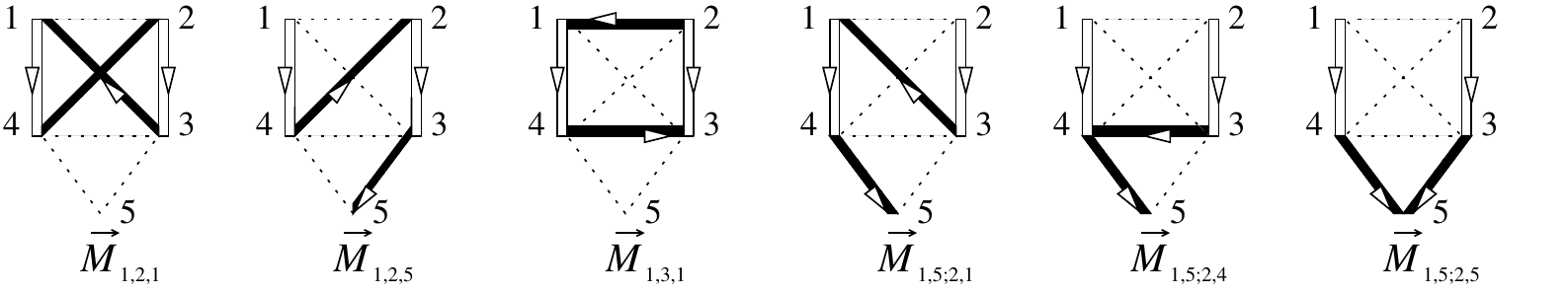} 
\caption{Output of the complete algorithm.}\label{fig:fig5}
\end{center}
\end{figure}

The \emph{weight} of $\T$ is the sum of the weights of the configurations it
contains. If the initial superimposition $M_0\cup M$ consists of cycles
only, \emph{i.e.} the set $\D$ is empty, then $\T=\{M\}$, and 
\begin{equation*}
w(\T)=w(M)=\sgn(\sigma_{M_0(M)})(-1)^{|\C|}\prod_{e\in M}a_e.
\end{equation*}
In all other cases:
\begin{equation*}
w(\T)=\sum_{j=1}^T
\sum_{
(\gamma_{{\ell}_1^1},\dots,\gamma_{{\ell}_1^j})
\in \Gamma_j} 
w(M_{\gamma_{{\ell}_1^1},\dots,\gamma_{{\ell}_1^j}}).
\end{equation*}
Since for every $j\in\{1,\dots,T\}$, and for every
$(\gamma_{{\ell}_1^1},\dots,\gamma_{{\ell}_1^j})\in\Gamma_j$, the
superimposition
$M_0\cup M_{\gamma_{{\ell}_1^1},\dots,\gamma_{{\ell}_1^j}}$ contains no
doubled edge of $\D$, we have:
\begin{equation}\label{equ:weight1}
w(M_{\gamma_{{\ell}_1^1},\dots,\gamma_{{\ell}_1^j}})=\sgn(\sigma_{M_0(M_{
\gamma_{{\ell}_1^1
},\dots,\gamma_{{\ell}_1^j}})})(-1)^{|\C|
} \prod_{e\in
M_{\gamma_{{\ell}_1^1},\dots,\gamma_{{\ell}_1^j}}}a_{e}.
\end{equation}

By iterating the argument of Section \ref{sec:weight}, we obtain the
following:
\begin{cor}\label{cor:2}
The complete algorithm is weight preserving, that is:
$$
w(\T)=w(M).
$$ 
\end{cor}

\subsection{Geometric characterization of
configurations}\label{sec:sec26}

Consider the superimposition $M_0\cup M$, recall that
$\C$ denotes the set of
cycles of length $\geq 4$ of $M_0\cup M$, and that $\D$ denotes its set of
doubled edges. Consider the complete algorithm with initial
superimposition $M_0\cup M$ in the case where $M_0\cup M$ contains doubled
edges, that is, when $\D\neq\emptyset$. Let
$j\in\{1,\dots,T\}$,
$(\gamma_{{\ell}_1^1},\dots,\gamma_{{\ell}_1^j})\in\Gamma_j$, and
$M_{\gamma_{{\ell}_1^1},\dots,\gamma_{{\ell}_1^j}}\in\T$ be a generic output;
and
denote by $F_{\gamma_{{\ell}_1^1},\dots,\gamma_{{\ell}_1^j}}$ the
superimposition
$M_0 \cup M_{\gamma_{{\ell}_1^1},\dots,\gamma_{{\ell}_1^j}}$. In order to
simplify
notations, we introduce:
\begin{equation*}\label{equ:F}
\forall i\in\{1,\dots,j\},\quad F_i:=
F_{\gamma_{{\ell}_1^1},\dots,\gamma_{{\ell}_1^i}}.
\end{equation*} 
One should keep in mind that the index $j$ refers to the last step of the
algorithm, and that indices $i\in\{1,\dots,j-1\}$ refer to intermediate steps.
As a consequence of the algorithm, see \eqref{equ:gamma1}, the configuration
$F_j$ and the paths
$\gamma_{{\ell}_1^1},\dots,\gamma_{{\ell}_1^j}$ satisfy the following
properties.

$\bullet$ The oriented edge configuration $F_j$:
\begin{enumerate}
\item[(I)] has one outgoing edge at every vertex of $V$. It consists of the
paths $\gamma_{{\ell}_1^1},\dots,\gamma_{{\ell}_1^j}$ and of the cycles $\C$ of
the initial superimposition $M_0\cup M$;
\item[(II)] has no doubled edge of $\D$ since the complete algorithm precisely
stops when this is the case.
\end{enumerate}
For every $i\in\{1,\dots,j\}$, the path $\gamma_{{\ell}_1^i}$ satisfies the
following.
\begin{enumerate}
 \item[(III)] It has even length $2k_i$ for some $k_i\in\{1,\dots,m_i\}$ and is
alternating. It starts from the vertex ${\ell}_1^{i}$, followed by an edge of
$M_0$.
\item[(IV)] The vertex ${\ell}_1^i$ is the smallest vertex
belonging to a doubled edge of
$F_{i-1}$ (understood as
$M_0\cup M$ when $i=1$).
The $2k_i$ first vertices are all distinct and the last vertex
${\ell}_{k_i+1}^i$ belongs to:
$$\Bigl(V_{{\ell}_{k_i}^i}^{\D_{\gamma_{{\ell}_1^1},\dots,\gamma_{{\ell}_1^{i-1
}},{\ell}_1^i,
\dots , {\ell}_{k_i}^i}}\Bigr)^c=
V_{{\ell}_{k_i}^i}\cap\bigl\{R\cup V^{\C}\cup
V^{\gamma_{{\ell}_1^1},\dots,\gamma_{{\ell}_1}^{i-1}}\cup
\{{\ell}_1^i,({\ell}_1^i)',\dots,{\ell}_{k_i}^i,({\ell}_{k_i}^i)'\}.
$$
As a consequence, one of the following 5 cases holds. 

$\bullet$ If ${\ell}_{k_i+1}^i\in
R\cup
V^{\C}\cup\{{\ell}_1^i,({\ell}_1^i)',\dots,{\ell}_{k_i}^i,({\ell}_{k_i}^i)'\}
\bigr\}$, then $\gamma_{{\ell}_1^i}$ consists of a new connected component, and
we
recover the three cases obtained
after Step 1 of the algorithm, replacing $\gamma_{{\ell}_1}$ by
$\gamma_{{\ell}_1^i}$, see Section \ref{sec:output}. For convenience of the
reader, we repeat these cases here.
\begin{enumerate}
\item[(IV)(1)] If ${\ell}_{k_i+1}^i\in R$, then $\gamma_{{\ell}_1^i}$ is a
loopless
oriented path from ${\ell}_1^i$ to one of the root vertices of $R$. Since
$R=\{n+1,\dots,n+r\}$, ${\ell}_1^i$ is smaller than all vertices of
~$\gamma_{{\ell}_1^i}$. The vertex ${\ell}_1^i$ is a leaf of a connected
component of
$F_i$, which consists of the path $\gamma_{{\ell}_1^i}$.
 
\item[(IV)(2)] If ${\ell}_{k_i+1}^i\in V^{\C}$, then $\gamma_{{\ell}_1^i}$ is a
loopless oriented path ending at a vertex of one of the cycles of $\C$. The
vertex ${\ell}_1^i$ is smaller than the $2k_i$ first vertices of the path, but
cannot be compared to vertices of the cycle of $\C$. By construction of the
orientation of $M_0\cup M$, see Section \ref{sec:sec21},
the orientation of the cycle is compatible with that of the edge $(i_1,i_2)$,
where $i_1$ is the smallest vertex of the cycle and $i_2$ is its partner in
$M_0$. The vertex ${\ell}_1^i$ is a leaf of a connected component of $F_i$,
which is a unicycle with $\gamma_{{\ell}_1^i}$ as unique branch and a cycle of
$\C$
as cycle.

\item[(IV)(3)]
When ${\ell}_{k_i+1}^i\in
\{{\ell}_1^i,({\ell}_1^i)',\dots,{\ell}_{k_i}^i,({\ell}_{k_i}^i)'\}$,
then $\gamma_{{\ell}_1^i}$ contains a loop of length~$\geq~3$. Recall that
configurations with odd
cycles cancel because of the skew-symmetry of the matrix, so that we only
consider configurations where ${\ell}_{k_i+1}={\ell}_s^i$ for some
$s\in\{1,\dots,k_i\}$. In this case, the part of the path
$\gamma_{{\ell}_1^i}$ to the loop has even length, is alternating and start
with an edge of $M_0$. The loop has even length $\geq 4$, is alternating and its
orientation is compatible with the
orientation of the edge $({\ell}_s^i,{{\ell}_s^i}^{'})$. The vertex ${\ell}_1^i$
is smaller
than all vertices of the path to the cycle and smaller than all vertices of the
cycle.

If ${\ell}_{k_i+1}^i\neq {\ell}_1^i$, then ${\ell}_1^i$ is a leaf of a connected
component
of $F_i$ which is a unicycle with a unique branch, consisting of the
path $\gamma_{{\ell}_1^i}$.

If ${\ell}_{k_i+1}^i={\ell}_1^i$, then ${\ell}_1^i$ is the smallest vertex of a
connected
component of $F_i$ which is a cycle, consisting of the path
$\gamma_{{\ell}_1^i}.$ 
\end{enumerate}

$\bullet$
If ${\ell}_{k_i+1}^i\in V^{\gamma_{{\ell}_1^1},\dots,\gamma_{{\ell}_1^{i-1}}}$
then, the path $\gamma_{{\ell}_1^i}$ attaches itself to a connected component of
$F_{i-1}$, this can only occur
when $i\geq 2$, and one of the following happens. 
\begin{enumerate}
\item[(IV)(4)] The path
$\gamma_{{\ell}_1^i}$ attaches itself to a leaf of $F_{i-1}$, that is,
${\ell}_{k_j+1}={\ell}_1^t$ for some $t\in\{1,\dots,i-1\}$. Then
$\gamma_{{\ell}_1}^i$ is a loopless oriented path from ${\ell}_1^i$ to
${\ell}_1^t$. The vertex ${\ell}_1^i$ is
smaller than the $2k_i$ following ones, but greater than than ${\ell}_1^t$.
Indeed
${\ell}_1^t$ is the starting point of a previous step of the algorithm.
\emph{This
allows to identify the ending vertex of the path $\gamma_{{\ell}_1}^i$}. The
vertex ${\ell}_1^i$ is a leaf of $F_i$.
\item[(IV)(5)] The path $\gamma_{{\ell}_1^i}$ creates a new branch of the
component.
Then $\gamma_{{\ell}_1^i}$ is a loopless oriented path. The vertex ${\ell}_1^i$
is
smaller than the $2k_i$ following ones, but
we have no a priori information on the last vertex of the path. The last
vertex of the path $\gamma_{{\ell}_1^i}$ is nevertheless identified as being a
fork.
The vertex ${\ell}_1^i$ is a leaf of $F_i$. Note that the
component of ${\ell}_1^i$ might be a unicycle with a unique branch. If this is
the case, the branch is the path $\gamma_{{\ell}_1^i}$ and the cycle was created
by
Case (IV)(3) in a previous step of the algorithm. The vertex ${\ell}_1^i$ is
thus
larger than the smallest vertex of the cycle.
\end{enumerate}
\end{enumerate}

\begin{lem}\label{lem:N1}
The oriented edge configuration $F_j$ is an RCRSF
compatible with~$M_0$.
\end{lem}
\begin{proof}
By definition of the algorithm, the oriented edge configuration $F_j$ contains
as many edges as $M_0\cup M$, that is $|V|$ edges. By
definition, it contains all edges of $M_0$, that is $|V|/2$ edges, and by the
algorithm no doubled edges of $\D$, that is $|V|/2$ edges not in $M_0$.

By Point (I) the oriented edge configuration $F_j$ has one outgoing edge at
every vertex of $V$, which is equivalent to saying that it is an RCRSF such
that edges of each component are oriented towards its root, and
cycles are oriented in one of the two possible directions. It thus remains to
show that cycles of unicycles are alternating, and have even length $\geq 4$.
By Point (II), the oriented edge configuration $F_j$ has no doubled edge of
$\D$, thus if $F_j$ has a cycle, it 
either comes from Point (IV)(2) meaning that it is a cycle of $\C$ implying that
it is even, alternating and has length $\geq 4$; or from Point
(IV)(3), when it is created by the algorithm. Returning to the description
of Point (IV)(3) and recalling that the contribution of configurations with
odd cycles cancel, we know that it has
the same properties in this case.
\end{proof}

For every $i\in\{1,\dots,j\}$, and for every connected component of $F_i$ which
is a cycle $C$ created by the
algorithm (\emph{i.e.} not a cycle of the
initial superimposition), denote by $m_C$
the smallest vertex of $C$. Define $x_i$ to be:
$$
x_i=
\begin{cases}
\max\{m_C:\,C \text{ is a cycle-connected component of } F_i,\text{ but not of
}\C\}&\text{ if $\{\}\neq\emptyset$}\\
-\infty&\text{ otherwise}.
\end{cases}
$$
If $F_i$ has at least one leaf, let $y_i$ be the maximum leaf of
$F_i$, else let $y_i=-\infty$. 

If both $x_i$ and $y_i$ are $-\infty$, then $F_i$ has no leaves and only
contains cycles of the initial superimposition $M_0\cup M$. This means that
the set $\D$ is empty, and 
that $F$ is the initial superimposition $M_0\cup M$. 
This has been excluded
here, since the complete algorithm
doesn't even start the opening of edges procedure in this case. Thus
$\max\{x_i,y_i\}>-\infty$.

\begin{lem}\label{lem:charactInitial}
For every $i\in\{1,\dots,j\}$,
the initial vertex ${\ell}_1^i$ of Step $i$ is the maximum of $x_i$ and $y_i$. 
\end{lem}

\begin{proof}
By Point (IV) above, the vertex ${\ell}_1^i$ is either a leaf of $F_i$
or the smallest vertex of a connected component of $F_i$ which is a
cycle created by the algorithm, meaning that it is not a cycle of $\C$
\emph{i.e.} not a cycle of the initial superimposition $M_0\cup M$. Arguing by
induction, all leaves and smallest vertices of cycle-components of $F_i$
which are not present in $\C$,
must be initial vertices of steps $i$ of the algorithm for
some $i\in\{1,\dots,j\}$. Moreover by
construction, the vertex ${\ell}_1^i$ is larger than all previous initial steps
of
the algorithm, thus proving the lemma.
\end{proof}

Properties described in Point(IV) also characterize the path
$\gamma_{{\ell}_1^i}$
once the initial vertex ${\ell}_1^i$ is fixed. This can be summarized in the
following lemma.

\begin{lem}\label{lem:charactPath}
Let ${\ell}_1^i$ be the initial vertex of Step $i$. Then:
\begin{itemize}
\item Suppose that ${\ell}_1^i$ is a leaf of a connected component of
$F_i$. When the connected component is a unicycle rooted on a cycle
created by the algorithm, we assume moreover that it contains more than one
branch. Then, we are in Cases (IV)(1)(2) or (5)
and the path $\gamma_{{\ell}_1^i}$ is characterized as the subpath of 
the unique path from ${\ell}_1^i$ to the root of the connected component,
stopping
the first time one visits a vertex which: belongs to $R$ or to
the cycle of the component; is a fork; is smaller than ${\ell}_1^i$.
\item Suppose that ${\ell}_1^i$ is the leaf of a unicycle of $F_i$
rooted on a cycle
created by the algorithm and containing a unique branch. 
If ${\ell}_1^i$ is larger
than the smallest vertex of the cycle, then we are in Case (IV)(5) and
the path $\gamma_{{\ell}_1^i}$ is the
path from ${\ell}_1^i$ to the cycle, stopping when the cycle is reached. Else,
if
${\ell}_1^i$ is smaller than the smallest vertex of the cycle, we are in Case
(IV)(3) and the path
$\gamma_{{\ell}_1^i}$ is the path from ${\ell}_1^i$ to the cycle, followed by
the cycle,
with the orientation specified in~(IV)(3).
\item If ${\ell}_1^i$ is the smallest vertex of a connected component of
$F_i$ which is a cycle created by the algorithm, then we are in Case
(IV)(3) and the path
$\gamma_{{\ell}_1^i}$ is the cycle, with the orientation specified in~(IV)(3).
\end{itemize}
\end{lem}

\begin{rem}\label{rem:rem1}
If the initial superimposition $M_0\cup M$ consists of cycles only, that is, if
the set $\D$ is empty, then the output of the complete algorithm is $F=M_0\cup M$,
which consists of alternating cycles of even length $\geq 4$. The orientation
of cycles is specified in Section~\ref{sec:sec21}, and cycles are oriented
in one of the possible two directions. In this case also, $F$ is an RCRSF
compatible with $M_0$. 
\end{rem}

\section{Proofs and corollaries}\label{sec:sec3}

We now prove Theorem \ref{thm:main}, Corollary \ref{cor:main} and state and
prove the line bundle version of the result.

\subsection{Proof of Theorem \ref{thm:main}}\label{sec:sec31}

In this section, we prove Theorem \ref{thm:main}. Let $M_0$ be a reference
perfect matching of $G$, and let $F$ be an RCRSF compatible with $M_0$,
containing $k_F$ unicycles. In Lemma \ref{lem:alpha}, we suppose that $F$ is an
output of the complete algorithm and identify $2^{k_F}$ possible perfect matchings
$M$ for
the initial superimposition $M_0\cup M$. Then, we introduce a partial reverse
algorithm used to define Condition (C) for RCRSFs compatible with $M_0$. In
Proposition \ref{thm:thm2}, we prove that an RCRSF compatible with $M_0$ is an
output of the complete algorithm if and only if it satisfies Condition (C), and if
this is the case, it is obtained $2^{k_F}$ times. The remainder
of the proof consists in showing that contribution of RCRSFs containing
unicycles cancel, and that only spanning forests remain with the appropriate
weight, thus proving Theorem~\ref{thm:main}.

Let $F$ be an RCRSF compatible with $M_0$, and let $k_F$ denote the number of
unicycles it contains. If
$k_F\neq 0$, we let $\{C_1,\dots,C_{k_F}\}$ be its set of cycles. For every
$(\eps_1,\dots,\eps_{k_F})\in\{0,1\}^{k_F}$,
define the edge configuration $M^{(\eps_1,\dots,\eps_{k_F})}$ as follows:
$$
M^{(\eps_1,\dots,\eps_{k_F})}=
\begin{cases}
\text{edges of $M_0$ on branches of $F$}\\
\text{edges of $M_0$ on the cycle $C_j$, when $\eps_j=0$}\\
\text{edge of $F\setminus M_0$ on the cycle $C_j$, when $\eps_j=1$}. 
\end{cases}
$$
If $k_F=0$, then the set of cycles of $F$ is $\{\emptyset\}$, and we set
$M^{(\eps_1,\dots,\eps_{k_F})}=M_0$.

\begin{lem}\label{lem:epsilonDimer}
For every \emph{RSCRSF} $F$ compatible with $M_0$, and
every $(\eps_1,\dots,\eps_k)\in\{0,1\}^{k_F}$, the edge
configuration $M^{(\eps_1,\dots,\eps_{k_F})}$ is a perfect matching of
$G$. 
\end{lem}
\begin{proof}
If $F$ contains no unicycles, $M^{(\eps_1,\dots,\eps_{k_F})}=M_0$ and this is
immediate. Suppose $k_F~\neq~0$.
Since $M_0$ is a perfect matching, and since the restriction of $M_0$ and
the restriction of $M^{(\eps_1,\dots,\eps_k)}$ to branches of $F$ are the same,
all vertices of $V\setminus\{V(C_1),\dots,V(C_{k_F})\}$ are incident to
exactly one
edge of $M^{(\eps_1,\dots,\eps_{k_F})}$. Moreover, by assumption for every
$j$,
the cycle $C_j$ is alternating, implying that each vertex of $V(C_j)$ is
incident to exactly one edge of the restriction of $M_0$ and one edge of the
restriction of $F\setminus M_0$ to $C_j$. As a consequence, every vertex of $V$
is incident to exactly one edge of $M^{(\eps_1,\dots,\eps_{k_F})}$, proving
that
it is a perfect matching of $G$.
\end{proof}

\begin{lem}\label{lem:alpha}
Let $F$ be the superimposition of $M_0$ and of an output of the complete
algorithm. Then, the perfect matching $M$
of the initial superimposition $M_0\cup M$, must be equal to
$M^{(\eps_1,\dots,\eps_{k_F})}$ for some
$(\eps_1,\dots,\eps_{k_F})\in\{0,1\}^{k_F}$.
\end{lem}
\begin{proof}
If $F$ is an output of the complete algorithm, then by Lemma \ref{lem:N1} and Remark
\ref{rem:rem1}, it is an
RCRSF compatible with $M_0$, so that $\forall
(\eps_1,\dots,\eps_{k_F})\in\{0,1\}^{k_F}$, the perfect matching
$M^{(\eps_1,\dots,\eps_{k_F})}$ is well defined. Suppose that
$F$ is an output of the complete algorithm with initial
superimposition $M_0\cup M$, where $M$ is not $M^{(\eps_1,\dots,\eps_{k_F})}$
for
some $(\eps_1,\dots,\eps_{k_F})\in\{0,1\}^{k_F}$. Then,
$M_0\cup M$ contains at least one cycle $C$ which is not $C_1,\dots,C_{k_F}$.
Returning to the definition of the algorithm, we know that cycles present in
the initial superimposition are also present in the output $F$. This
yields a contradiction since $F$ contains exactly the cycles
$C_1,\dots,C_{k_F}$.
\end{proof}

\underline{\textbf{Partial reverse algorithm}}

\underline{Input}: an RCRSF $F$ compatible with $M_0$ not consisting of cycles
only.

\underline{Initialization}: $F_1=F$. 

\underline{Step $i$, $i\geq 1$}

Let $\bar{\ell}_1^i$ be the largest leaf of $F_i$, and consider the connected
component containing $\bar{\ell}_1^i$. Start from $\bar{\ell}_1^i$ along
the unique path joining $\bar{\ell}_1^i$ to
the root or the cycle of the component, until the
first time one of the following vertices is reached:
\begin{itemize}
\item the root vertex if the component is a tree, or the cycle if it is a
unicycle;
\item a fork;
\item a vertex which is smaller than the leaf $\bar{\ell}_1^i$.  
\end{itemize}
This yields a loopless path $\lambda_{\bar{\ell}_1^i}$ starting from
$\bar{\ell}_1^i$, of length $\geq 1$.
Let $F_{i+1}=F_{i}\setminus\lambda_{\bar{\ell}_1^i}$. If $F_{i+1}$ is
empty or contains cycles only, then stop. Else, go to Step $i+1$.

\underline{End}: since edges are removed at every step and since $F$ contains
finitely
many edges, the algorithm ends in finite time $N$.

\begin{defi}\label{def:N2}
An RCRSF $F$ compatible with $M_0$ is said to satisfy \emph{Condition} (C) if
either $F$ consists of cycles only, or if each of the
paths $\lambda_{\bar{\ell}_1^1},\dots,\lambda_{\bar{\ell}_1^N},$ obtained from
the
partial reverse algorithm has even length and starts from an edge of
$M_0$. 
\end{defi}

\begin{prop}\label{thm:thm2}
Let $F$ be an RCRSF compatible with $M_0$.\\
Then, for every $(\eps_1,\dots,\eps_{k_F})\in\{0,1\}^{k_F}$, $F$ is the
superimposition of $M_0$ and of an output of the complete algorithm, with initial
superimposition $M_0~\cup~
M^{{(\eps_1,\dots,\eps_{k_F})}}$,
if and only if $F$ satisfies Condition \emph{(C)}. The
orientation of cycles of $F$ is specified by the proof.
\end{prop}

\begin{proof}
Let $F$ be an RCRSF compatible with $M_0$ containing $k_F$ unicycles,
and denote by $\{C_1,\dots,C_{k_F}\}$ its set of cycles. For every $
(\eps_1,\dots,\eps_{k_F})\in\{0,1\}^{k_F}$, the
edge configuration
$M^{{(\eps_1,\dots,\eps_{k_F})}}$ is well defined, and by Lemma
\ref{lem:epsilonDimer} is a perfect matching. 
We now fix $(\eps_1,\dots,\eps_{k_F})\in\{0,1\}^{k_F}$. Recall that if
$k_F=0$, then the perfect matching $M^{{(\eps_1,\dots,\eps_{k_F})}}$ is simply
$M_0$.

In the case where $k_F\neq 0$, $(\eps_1,\dots,\eps_{k_F})=(1,\dots,1)$, and
$F$ consists of cycles
only, the superimposition $M_0\cup M^{(\eps_1,\dots,\eps_{k_F})}$
consists of cycles only, and $F=M_0\cup M^{(\eps_1,\dots,\eps_{k_F})}$ is an
output of the complete algorithm. The orientation of the cycles is specified by the
choice of orientation of Section \ref{sec:sec21}, thus proving Proposition
\ref{thm:thm2}.

Assume that we are not in the above case. Then $F$ is an output of the
algorithm with initial superimposition $M_0\cup
M^{(\eps_1,\dots,\eps_{k_F})}$ if and only if there exists a positive integer
$j$
and a sequence of paths
$(\gamma_{{\ell}_1^1},\dots,\gamma_{{\ell}_1^j})\in\Gamma_j$
such that $F=M_0\cup
M_{\gamma_{{\ell}_1^1},\dots,\gamma_{{\ell}_1^j}}^{(\eps_1,\dots,\eps_{k_F})}
$.

Lemmas \ref{lem:charactInitial} and \ref{lem:charactPath} give a
characterization of ${\ell}_1^i$ and $\gamma_{{\ell}_1^i}$ at every step of the
algorithm. This allows us to define a complete reverse algorithm.

\underline{\textbf{Complete reverse algorithm}}

\underline{Input}: an RCRSF $F$ compatible with $M_0$, 
$(\eps_1,\dots,\eps_{k_F})\in\{0,1\}^{k_F}$ as above, and the corresponding
perfect matching $M^{(\eps_1,\dots,\eps_{k_F})}$. If $k_F=0$, then
$M^{(\eps_1,\dots,\eps_{k_F})}=M_0$. 

\underline{Initialization}: $F_1=F$.

\underline{Step $i$, $i\geq 1$}.

Since $F_i$ is either $F_1$ or is obtained from $F_{i-1}$ by removing edges,
the set of cycles of $F_i$ is included in the set of cycles
$\{C_1,\dots,C_{k_F}\}$ of $F_1$.

For every connected component of $F_i$ which is a cycle $C_{\alpha}$ such that
$\eps_\alpha=0$ (meaning that $C_\alpha$ is not a cycle of the initial
superimposition $M_0\cup M^{(\eps_1,\dots,\eps_{k_F})}$), let $m_{C_\alpha}$ 
be the smallest vertex of $C_\alpha$. Define
$$
x_i=
\begin{cases}
\max\{m_{C_\alpha}:\,C_\alpha\text{ is a cycle-connected component of }F_i
\text{, and }\eps_\alpha=0\}&\text{ if $\{\}\neq\emptyset$}\\
-\infty&\text{ else}.
\end{cases}
$$
If $F_i$ has at least one leaf, let $y_i$ be the maximum leaf, else let
$y_i=-\infty$.
Note that by assumption, we do not have $x_i=y_i=-\infty$. We 
let ${\ell}_1^i=\max\{x_i,y_i\}$, and $\gamma_{{\ell}_1^i}$ be the oriented path
as
characterized in Lemma \ref{lem:charactPath}.

Let $F_{i+1}=F_i\setminus\gamma_{{\ell}_1^i}$. If the oriented edge
configuration $F_{i+1}$ is empty, or if it consists of cycles of the
superimposition 
$M_0\cup M^{(\eps_1,\dots,\eps_{k_F})}$ only, then stop; else go to Step $i+1$.

\underline{End}: since edges are removed at every step and since $F$ contains
finitely
many edges, the algorithm ends in finite time $j$, for some integer $j$.

This defines for every RCRSF $F$ compatible with $M_0$, a sequence of paths
$\gamma_{{\ell}_1^1},\dots,\gamma_{{\ell}_1^j}$ such that $F$ is the union of
these paths and of cycles of the initial superimposition $M_0\cup
M^{(\eps_1,\dots,\eps_{k_F})}$. As a consequence, the oriented edge
configuration $F$ satisfies Properties (I), (II), (IV)(1)$-$(5). We are thus
left with
proving that $F$ satisfies Property (III) if and only if it satisfies
Condition (C), \emph{i.e.} we need to show that the paths
$(\gamma_{{\ell}_1^1},\dots,\gamma_{{\ell}_1^j})$ all have even length, are
alternating and start from an edge of $M_0$ if and only if $F$ satisfies
Condition (C).

Observe that initial vertices $({\ell}_1^1,\dots,{\ell}_1^j)$ of the complete
reverse
algorithm consist of initial vertices
$( {\bar{\ell}}_1^1,\dots,{\bar{\ell}}_1^{N})$ of the partial reverse
algorithm, interlaced with smallest vertices of
components which are cycles. Indeed, the only difference in the partial
reverse algorithm is that cycles are not removed, but this does not change the
characterization of largest leaf.

If ${\ell}_1^i$ is the smallest vertex of a component of $F_i$ which is a cycle,
that is, if ${\ell}_1^i=x_i$, then
$\gamma_{{\ell}_1^i}$ is a cycle $C_\alpha\in\{C_1,\dots,C_{k_F}\}$ such that
$\eps_\alpha=0$.
Since $F$ is compatible with $M_0$, the cycle has even length and is
alternating. The orientation is fixed by the algorithm and $\gamma_{{\ell}_1^i}$
always satisfies Property (III).

If ${\ell}_1^i$ is the largest leaf of $F_i$, that is, if ${\ell}_1^i=y_i$,
then in all cases except one, which we treat below, the path
$\gamma_{{\ell}_1^i}$
is exactly the path $\lambda_{\bar{\ell}_1^{i'}}$ of the partial reverse
algorithm, for some $i'\leq i$.
Condition (C)
says that $\lambda_{\bar{\ell}_1^{i'}}$ has even length and starts from an
edge
of $M_0$.
In order to show that this is equivalent to satisfying Property (III), we are
left with showing that, by construction, the path
$\lambda_{\bar{\ell}_1^{i'}}$ is
always
alternating. Suppose that this is not the case, then there are at least two
edges of the same kind (either in $M_0$ or not in $M_0$) which follow each other.
This implies that there is a vertex $v$ of the path incident to two
edges of the same kind. Since $M_0$ is a perfect matching, every vertex is
incident to exactly one edge of $M_0$, so that we cannot have two edges of
$M_0$ following each other. Thus these two edges do not belong to $M_0$. Again,
since $M_0$ is a perfect matching, the vertex $v$ is also incident to an
edge of $M_0$, implying that $v$ is the end of a branch. By construction of the
path $\lambda_{\bar{\ell}_1^{i'}}$, the path must stop at $v$, implying that
one of
the
two edges is not in $\lambda_{\bar{\ell}_1^{i'}}$, yielding a contradiction.

We now treat the last case. If ${\ell}_1^i$ is a leaf of a connected component
of
$F_i$ which is a unicycle rooted on a cycle $C_\alpha$ such that
$\eps_\alpha=0$,
with a unique branch, and such that ${\ell}_1^i$ is smaller than the smallest
vertex of the cycle. Then the path $\gamma_{{\ell}_1^i}$ is the path
$\lambda_{\bar{\ell}_1^{i'}}$ followed by the cycle with the appropriate
orientation. We
have to show that $\gamma_{{\ell}_1^i}$ satisfies Property (III) if and only if
$\lambda_{\bar{\ell}_1^{i'}}$ satisfies Condition (C). 
By Property (III), we know that the part of $\gamma_{{\ell}_1^i}$ stopping when
the
cycle is reached, which is precisely $\lambda_{\bar{\ell}_1^{i'}}$, has even
length
and
starts from an edge of $M_0$. This is exactly Condition (C), since by the
same argument as above, the path $\lambda_{\bar{\ell}_1^{i'}}$ is
alternating. We
conclude
by observing that since $F$ is compatible with $M_0$, the cycle part of
$\gamma_{{\ell}_1^i}$ is alternating, and
starts from an edge of $M_0$ by construction of the orientation of
the cycle. Thus $\gamma_{{\ell}_1^i}$ satisfies Property (III) if and only if
$F$ is compatible with $M_0$ and satisfies Condition (C).
\end{proof}

We denote by $\GG(M_0)$ the set of RCRSFs compatible with $M_0$ satisfying
Condition (C). Let $F$ be an RCRSF of $\GG(M_0)$, and let $k_F$ be its number of
unicycles. If $k_F\neq 0$, then for every
$(\eps_1,\dots,\eps_{k_F})\in\{0,1\}^{k_F}$, denote by
$F^{(\eps_1,\dots,\eps_{k_F})}$ the version of $F$ obtained from the complete
algorithm with initial superimposition $M_0 \cup
M^{(\eps_1,\dots,\eps_{k_F})}$, with the orientation of cycles given by
Proposition \ref{thm:thm2}. If $k_F=0$, then $F$ is obtained exactly once from
the complete algorithm with initial superimposition $M_0 \cup M_0$. 

Since $M_0 \cup M^{(\eps_1,\dots,\eps_{k_F})}$ has exactly
$\sum_{i=1}^{k_F}\eps_i$ cycles, and since $M_0\cup M_0$ has none, we have as
a consequence of the complete algorithm, see Equation \eqref{equ:weight1}, that the
weight $w_{M_0}(F^{(\eps_1,\dots,\eps_{k_F})}\setminus M_0)$ is equal to:
\begin{equation}\label{equ:final1}
\sgn(\sigma_{M_0(F^{(\eps_1,\dots,\eps_{k_F})}\setminus M_0)})\cdot
\begin{cases}
(-1)^{\sum_{i=1}^{k_F}\eps_i}
\prod\limits_{e\in F^{(\eps_1,\dots,\eps_{k_F})}\setminus M_0}a_e& \text{ if
$k_F\neq 0$}\\
\prod\limits_{e\in F^{(\eps_1,\dots,\eps_{k_F})}\setminus M_0}a_e& \text{ if
$k_F=0$}.
\end{cases}
\end{equation}

Recall that $\T$ denotes the output of the complete algorithm with initial
superimposition $M_0\cup M$ for a fixed reference perfect matching $M_0$ and a
generic perfect matching $M$ of $G$. Since we now aim at taking the union over
all
perfects matchings $M$, we write $\T$ as
$\T_{M_0}(M)$.

As a consequence of Proposition \ref{thm:thm2}, we have that $\bigcup_{M\in
\M}\T_{M_0}(M)$ is equal to:
\begin{equation}\label{equ:final}
\Bigl(
\bigcup\limits_{\{F\in \GG(M_0):\,k_F\neq 0\}}
\bigcup\limits_{(\eps_1,\dots,\eps_k)\in\{0,1\}^{k_F}}
F^{(\eps_1,\dots,\eps_k)}\setminus M_0
\Bigr)
\bigcup
\Bigl(
\bigcup\limits_{\{F\in \GG(M_0):\,k_F=0\}}
F\setminus M_0
\Bigr).
\end{equation}
Returning to the definition of
the Pfaffian of Equation \eqref{equ:pfaffian}, using Corollary \ref{cor:2} and
Equation \eqref{equ:final} in the last line, we deduce that:
\begin{align*}
\Pf(A)&=\sum_{M\in\M}w_{M_0}(M), \text{
(Definition of Equation \eqref{equ:pfaffian})}\\
&=\sum_{M\in \M}w(\T_{M_0}(M)),\text{ (by Corollary \ref{cor:2})}\\
&=\underbrace{\sum_{\{F\in \GG(M_0):\,k_F\neq 0\}}
\sum_{(\eps_1,\dots,\eps_{k_F})\in\{0,1\}^{k_F}}
w_{M_0}(F^{(\eps_1,\dots,\eps_{k_F})}\setminus M_0)}_{(I)}+
\underbrace{\sum_{\{F\in \GG(M_0):\,k_F= 0\}}
w_{M_0}(F\setminus M_0)}_{(II)}.
\end{align*}
Let us show that $(I)$ is equal to zero. As a consequence of Equation
\eqref{equ:final1}, it is equal to:
\begin{align*}
(I)&=\sum\limits_{\{F\in \GG(M_0):\,k_F\neq 0\}} 
\sum_{(\eps_1,\dots,\eps_{k_F})\in\{0,1\}^{k_F}}
\sgn(\sigma_{M_0(F^{(\eps_1,\dots,\eps_{k_F})}\setminus
M_0)})(-1)^{\sum_{i=1}^{k_F}
\eps_i}
\prod_{e\in
F^{(\eps_1,\dots,\eps_{k_F})}\setminus M_0} a_{e}.
\end{align*}
Observing that the term
$\sgn(\sigma_{M_0(F^{(\eps_1,\dots,\eps_{k_F})}\setminus
M_0)})\prod_{e\in
F^{(\eps_1,\dots,\eps_{k_F})}\setminus M_0} a_{e}$
is independent of $(\eps_1,\dots,\eps_{k_F})$, we conclude that:
\begin{align*}
(I)&= \sum\limits_{\{F\in \GG(M_0):\,k_F\neq 0\}}
\Bigl(\sgn(\sigma_{M_0(F^{(\eps_1,\dots,\eps_{k_F})}\setminus M_0)})
\prod_{e\in F^{(\eps_1,\dots,\eps_{k_F})}\setminus M_0} a_{e}
\Bigr)
\sum_{(\eps_1,\dots,\eps_{k_F})\in\{0,1\}^{k_F}}
(-1)^{\sum_{i=1}^{k_F}}\\
&=\sum\limits_{\{F\in \GG(M_0):\,k_F\neq 0\}}
\Bigl(\sgn(\sigma_{M_0(F^{(\eps_1,\dots,\eps_{k_F})}\setminus M_0)})
\prod_{e\in F^{(\eps_1,\dots,\eps_{k_F})}\setminus M_0} a_{e}
\Bigr)(1-1)^{k_F}\\
&=0.
\end{align*}
Thus,
\begin{align*}
\Pf(A)&=\sum_{\{F\in \GG(M_0):\,k_F= 0\}}
w_{M_0}(F\setminus M_0) \\
&=\sum_{\{F\in \GG(M_0):\,k_F= 0\}}
\sgn(\sigma_{M_0(F\setminus M_0)})
\prod_{e\in F\setminus M_0}a_e.
\end{align*}

The set $\{F\in\GG(M_0):\,k_F=0\}$ consists of RCRSFs compatible with $M_0$
containing no unicycles, and satisfying Condition (C). Observing that:
\begin{itemize}
 \item the set of RCRSFs compatible with $M_0$, containing no unicycle is
exactly the set of spanning forests of $G^R$ compatible with $M_0$ of Section
\ref{sec:11} of the introduction,
\item in the case of spanning forests, the partial reverse algorithm is exactly
the trimming algorithm of Section \ref{sec:11},
\item Condition (C) of Definition \ref{def:def0} and Condition
(C) of Definition \ref{def:N2} are the same in the case of spanning forests,
\item the permutation $\sigma_{M_0(F\setminus M_0)}$ obtained from the
algorithm is the permutation of Definition \ref{def:def2},
\end{itemize}
we deduce that $\{F\in\GG(M_0):\,k_F=0\}=\F(M_0)$, and thus conclude the
proof of Theorem \ref{thm:main}.

\textsc{Example}. Let us take the reference matching $M_0$ which followed us
throughout the paper, and consider the three possible perfect matchings
$M_1,M_2,M_3$ of the graph $G$ given in Figure \ref{fig:fig8}. Figure
\ref{fig:fig5} shows the output of the complete algorithm with initial
superimposition $M_0\cup M_1$. Since the superimpositions $M_0\cup M_2$ and
$M_0\cup M_3$ contain doubled edges only, the output of the algorithm with
these
respective initial superimpositions, are the configurations themselves. By the
Theorem \ref{thm:main}, configurations $M_2$ and $M_{1,2,1}$ have opposite
weights, and configurations $M_3$ and $M_{1,3,1}$ as well, so that their
contributions cancel in the Pfaffian. As a consequence,
signed weighted half-spanning trees counted by the Pfaffian of the matrix $A$
are those of Figure \ref{fig:fig10} below.

\begin{figure}[ht]
\begin{center}
\includegraphics[width=8cm]{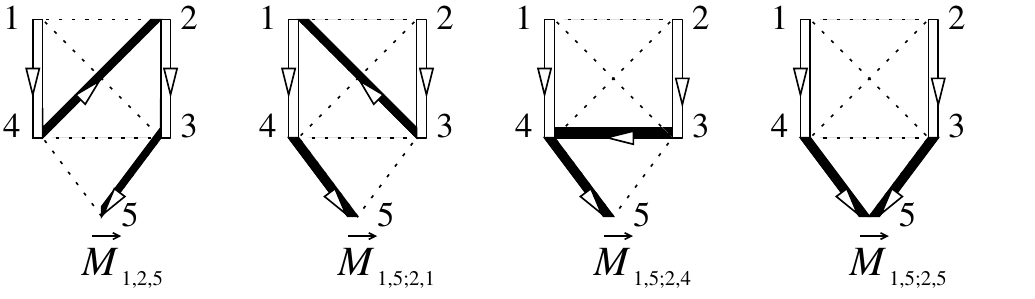} 
\caption{Black edges of the above configurations are
half-spanning trees counted by the Pfaffian of the matrix
$A$.}\label{fig:fig10}
\end{center}
\end{figure}

\subsection{Proof of Corollary \ref{cor:main}}\label{sec:sec32}

Let us recall the setting: $A^R$ is a skew-symmetric matrix of size
$(n+r)\times (n+r)$, whose column sum is zero, with $n$ even; $A$ is the matrix
obtained from $A^R$ by removing the $r$ last lines and columns; $G^R$ and $G$
are the graphs naturally constructed from $A^R$ and $A$ in the introduction.  

Recall, see Section \ref{sec:sec21}, that the sign of the permutation
$\sigma_{M}$ assigned to a perfect matching $M$ counted by the Pfaffian
of $A$, depends on the ordering of the two elements of pairs involved in the
perfect matching, but not on the ordering of the pairs themselves. Choosing the
sign of $\sigma_M$ thus amounts to choosing an orientation of edges of the
perfect matching $M$. The Pfaffian of $A$ can thus be written as:
\begin{equation*}
\Pf(A)=\sum_{M\in\M}\sgn(\sigma_{M})\prod_{e\in M}a_e, 
\end{equation*}
where the product is over coefficients corresponding to a choice of orientation
of edges of $M$, specifying a choice of permutation $\sigma_M$. 

Now, it is a known fact that the determinant of a skew-symmetric matrix is
equal to the square of the Pfaffian:
\begin{align*}
\det(A)&=\Bigl(\sum_{M_0\in \M}\sgn(\sigma_{M_0})\prod_{e\in M_0} a_e\Bigr)
\Bigl(\sum_{M\in \M}\sgn(\sigma_{M})\prod_{e\in M} a_e\Bigr).
\end{align*}
As in Section \ref{sec:sec21}, for every $M_0\in\M$, we choose the permutation
$\sigma_M$ using the superimposition $M_0\cup M$. Equation \eqref{equ:pfaffian}
thus yields:
\begin{equation*}
\det(A)=\sum_{M_0\in \M}\sgn(\sigma_{M_0})\prod_{e\in M_0} a_e
\Bigl(\sum_{M\in \M}\sgn(\sigma_{M_0(M)})(-1)^{|\D(M_0\cup M)|}(-1)^{|\C(M_0\cup
M)|}\prod_{e\in M}a_e \Bigr).
\end{equation*}
As a consequence of Theorem \ref{thm:main}, this can be rewritten as:
\begin{align*}
\det(A)&=\sum_{M_0\in \M}\sgn(\sigma_{M_0})\prod_{e\in M_0} a_e
\Bigl(\sum_{F\in \F(M_0)}\sgn(\sigma_{M_0(F\setminus M_0)})
\prod_{e\in F\setminus M_0}a_e \Bigr),\\
&=\sum_{M_0\in \M}\sum_{F\in \F(M_0)}\Bigl(\sgn(\sigma_{M_0})\prod_{e\in M_0}
a_e\Bigr)\sgn(\sigma_{M_0(F\setminus M_0)})
\prod_{e\in F\setminus M_0}a_e.
\end{align*}
where $\sigma_{M_0(F\setminus M_0)}$ is defined in Definition \ref{def:def2}.

We have not yet chosen the permutation $\sigma_{M_0}$ assigned to $M_0$, we do
so now. For every $F\in \F(M_0)$, we chose the orientation of $M_0$ to be the
orientation of edges induced by the spanning forest $F$: this is precisely
$\sigma_{M_0(F\setminus M_0)}$. Combining the product of coefficients $a_e$ over
oriented edges in $M_0$ and in $F\setminus M_0$ yields:
\begin{align*}
\det(A)&=\sum_{M_0\in \M}\sum_{F\in \F(M_0)}
\sgn(\sigma_{M_0(F\setminus M_0)})^2
\prod_{e\in F} a_e.\\
&=\sum_{M_0\in \M}\sum_{F\in \F(M_0)}
\prod_{e\in F} a_e,
\end{align*}
thus proving Corollary \ref{cor:main}. 

\begin{rem}\label{rem:main}$\,$
\begin{enumerate}
\item We now give an intrinsic characterization of $\cup_{M_0\in\M}\F(M_0)$,
not using
reference perfect matchings.

Consider the trimming algorithm of Section
\ref{sec:11} applied to general spanning
forests of $G^R$ (not assuming that they are compatible with a reference
perfect matching $M_0$). Since the reference perfect matching is not used in the
algorithm, everything works out in the same way, and the algorithm yields
a sequence of paths $\lambda_{{\ell}_1}^1,\dots,\lambda_{{\ell}_1}^N$.  
This yields the following more general form of Definition \ref{def:def0}.

\begin{defi}
A spanning forest $F$ of $G^R$ is said to satisfy \emph{Condition} (C) if each
of the paths $\lambda_1^1,\dots,\lambda_1^N$ obtained from the trimming
algorithm has even length. Let $\F$ denote the set of spanning forests of $G^R$
satisfying Condition (C).
\end{defi}

\begin{lem}
\begin{equation*}
\F=\bigcup_{M_0\in\M}\F(M_0).
\end{equation*}
\end{lem}

\begin{proof}
By definition, we have the following immediate inclusion:
$\bigcup_{M_0\in\M}\F(M_0)~\subset~\F$. 

If $M_0$ and $M_0'$ are two distinct perfect matchings of $G$, then $\F(M_0)\cap
\F(M_0')=\emptyset$. Indeed suppose there exists a spanning forest $F$ in the
intersection. Then, it must be compatible with $M_0$ and $M_0'$, meaning that
it contains all edges of $M_0\cup M_0'$. Since $M_0$ and $M_0'$ are distinct,
the superimposition $M_0\cup M_0'$ must contain a cycle, yielding a
contradiction with the fact that $F$ is a spanning forest.

Thus it remains to show that given a spanning forest $F$ satisfying Condition
(C) there exists a perfect matching $M_0$ such that $F\in\F(M_0)$, meaning that
$F$ is compatible with $M_0$ and satisfies Condition (C) of Definition
\ref{def:def0}. Let $F$ be a spanning forest satisfying Condition (C), and let
$\lambda_{{\ell}_1^1},\dots,\lambda_{{\ell}_1^N}$ be the sequence of paths
obtained from
the trimming algorithm. For every $i\in\{1,\dots,N\}$, let
$M_0(\lambda_{{\ell}_1^i})$ consist of half of the edges of
$\lambda_{{\ell}_1^i}$ such
that $\lambda_{{\ell}_1^i}$ alternates between edges of
$M_0(\lambda_{{\ell}_1^i})$ and edges of $\lambda_{{\ell}_1^i}\setminus
M_0(\lambda_{{\ell}_1^i})$, starting from an edge of
$M_0(\lambda_{{\ell}_1^i})$. Let
$M_0=\cup_{i=1}^N M_0(\lambda_{{\ell}_1^i})$. Then $F$ is compatible with $M_0$
and
satisfies Condition (C) of Definition \ref{def:def0}. It remains to show that
$M_0$ is a perfect matching. The edge configuration consists of $|V|/2$ edges,
since by construction, it consists of half of the edges of a spanning forest.
Moreover, since each of the paths
$\lambda_{{\ell}_1^1},\dots,\lambda_{{\ell}_1^N}$
has even length, no vertex is incident to two edges of $M_0$, thus proving that
$M_0$ is a perfect matching. 
\end{proof}

As a consequence, Corollary \ref{cor:main} can be rewritten in the simpler
form:
\begin{cor}\label{cor:main2}
\begin{equation*}
\det(A)=\sum_{F\in\F}\prod_{e\in F}a_e. 
\end{equation*}
\end{cor}
\item Let $\Xi$ be the set of cycle coverings of the graph $G$ by cycles of
even length: a typical element $\xi\in\,\Xi$ is of the form
$\xi=(C_1,\dots,C_k)$ for some $k$. Then, since the matrix $A$ is
skew-symmetric, the determinant of $A$ is equal to:
\begin{align*}
\det(A)&=\sum\limits_{\xi=(C_1,\dots,C_k)\in\Xi}\;
\prod_{\{i:|C_i|\geq 4\}}(-1)
\Bigl(\prod_{e\in \overrightarrow{C_i}}a_{e}+
\prod_{e\in \overleftarrow{C_i}}a_e\Bigr)
\prod_{\{i:|C_i|=2\}}(-1)a_e a_{-e}\\
&=\sum\limits_{\xi=(C_1,\dots,C_k)\in\Xi}\;
\prod_{\{i:|C_i|\geq 4\}}(-2)
\bigl(\prod_{e\in \overrightarrow{C_i}}a_{e}\bigr)
\prod_{\{i:|C_i|=2\}}a_e^2.
\end{align*}

It is also possible to prove Corollary \ref{cor:main2} directly, without
passing through the Pfaffian, by applying the complete algorithm to doubled edges
of configurations counted by the determinant, and by taking into account all
edges instead of half of them.
\end{enumerate}
\end{rem}

\subsection{Line-bundle matrix-tree theorem for skew-symmetric
matrices}\label{sec:sec33}

In the whole of this section, we change notations slightly, and we let $A$ be a
skew-symmetric matrix of size $n\times n$, whose column sum is zero, with $n$
even; $G=(V,E)$ denotes the graph associated to the matrix $A$.

We now state a line-bundle version of the matrix-tree theorem for
skew-symmetric matrices of Corollary \ref{cor:main}, in the spirit of what is
done for the Laplacian matrix in \cite{Forman}, \cite{KenyonVectorBundle}, but
first we need a few definitions.

A $\CC$-bundle is a copy $\CC_v$ of $\CC$ associated to each vertex $v\in V$.
The \emph{total space} of the bundle is the direct sum $W=\oplus_{v\in
V}\CC_v$. A \emph{connection} $\Psi$ on $W$ is the choice, for each oriented
edge $(i,j)$ of $G$ of linear isomorphism $\psi_{i,j}:\CC_i\rightarrow\CC_j$,
with the property that $\psi_{i,j}=\psi_{j,i}^{-1}$; that is, we associate to
each oriented edge $(i,j)$ a non-zero complex number $\psi_{i,j}$ such that
$\psi_{i,j}=\psi_{j,i}^{-1}$. We
say that $\psi_{i,j}$ is the \emph{parallel transport} of the connection over
the edge $(i,j)$. The \emph{monodromy} of the connection around an oriented
cycle $\vec{C}$
is the complex number $\omega_{\vec{C}}=\prod_{e\in \vec{C}} \psi_{e}$. 

We consider the matrix $A^{\psi}$ constructed from the
matrix $A$ and the connection $\psi$:
\begin{equation*}
(A^{\psi})_{i,j}=a_{i,j}^{\psi}=a_{i,j}\psi_{i,j}.
\end{equation*}

A \emph{cycle-rooted spanning forest} of $G$, also denoted $CRSF$, is an
oriented edge configuration spanning vertices of $G$ such that each connected
component is a tree rooted on a cycle. In all that follows, we assume that
cycles have length $\geq 3$. Edges of branches of the trees are oriented towards
the cycle, and the cycle is oriented in one of the two possible directions. 

Consider the partial reverse algorithm of Section \ref{sec:sec31} applied to
a general CRSF $F$. Since the reference perfect matching plays no role in this
algorithm, everything
works out in the same way, and the algorithm yields a sequence of paths
$\lambda_{{\ell}_1^1},\dots,\lambda_{{\ell}_1^N}$, whose union corresponds to
branches
of $F$. 

\begin{defi}
A CRSF of $G$ is said to satisfy \emph{Condition} (C) if all of the
paths $\lambda_{{\ell}_1}^1,\dots,\lambda_{{\ell}_1^N}$ obtained from the
partial
reverse algorithm have even length. Let us
denote by $\GG$ the set of CRSFs satisfying Condition (C). Then,
\end{defi}

Then, for a generic CRSF $F$ of $G$, let us denote by $(C_1,\dots,C_k)$ its cycles.

\begin{cor}\label{thm:linebundle}
\begin{align*}
\det(A^{\psi})=\sum_{F\in \GG}
\Bigl(\prod_{\{e\in \text{\rm branch}(F)\}}a_e\Bigr)
\Bigl(\prod_{\{i:|C_i|\text{\rm is odd}\}}
&\prod_{e\in
\overrightarrow{C_i}}a_e[\omega_{\overrightarrow{C_i}}-\omega_{\overrightarrow{
C_i }}^{-1} ]
\Bigr)\cdot\\
&\cdot\Bigl(\prod_{\{i:|C_i|\text{\rm is even}\}}
\prod_{e\in
\overrightarrow{C_i}}a_e[2-\omega_{\overrightarrow{C_i}}-\omega_{\overrightarrow
{C_i}}^{-1}]\Bigr).
\end{align*}
\end{cor}

\begin{proof}

We expand the determinant of $A^{\Psi}$ using cycle decompositions, as we
have done for the determinant of $A$ in Point 2 of Remark \ref{rem:main}.
Since the matrix
$A^{\Psi}$ is not skew-symmetric, we cannot omit odd cycles, and we let
$\Xi$ be the set of cycle decompositions of the graph $G$, that is, the
set of coverings of the graph by disjoint cycles. A typical element of $\Xi$ can
be written as $\xi=\{C_1,\dots,C_k\}$, for some positive integer $k$. Then, the
determinant of the matrix $A^{\psi}$ is:
\begin{equation*}
\det(A^{\psi})=
\sum\limits_{\xi=(C_1,\dots,C_k)\in\Xi}\;
\prod_{\{i:|C_i|\geq 3\}}(-1)^{|C_i|+1}
\Bigl(\prod_{e\in \overrightarrow{C_i}}a_{e}\psi_e +
\prod_{e\in \overleftarrow{C_i}}a_e\psi_e\Bigr)
\prod_{\{i:|C_i|=2\}}\Bigl((-1)a_e\psi_e a_{-e}\psi_{-e}\Bigr).
\end{equation*}

Using the skew-symmetry of the matrix $A$ and the fact that
$\psi_{-e}=\psi_e^{-1}$ this yields:
\begin{align*}
\det(A^{\psi})=
\sum_{\xi=(C_1,\dots,C_k)\in \Xi}\;&
\prod_{\{i:|C_i|\text{ is odd}\}}
\Bigl(\prod_{e\in
\overrightarrow{C_i}}a_e[\omega_{\overrightarrow{C_i}}-\omega_{\overrightarrow{
C_i }}^{-1} ]\Bigr)\cdot\\
&\cdot\prod_{\{i:|C_i|\text{ is even}\geq 4\}}(-1)
\Bigl(\prod_{e\in
\overrightarrow{C_i}}a_e[\omega_{\overrightarrow{C_i}}+\omega_{\overrightarrow{
C_i}}^{-1} ]\Bigr)\cdot
\prod_{\{i:|C_i|=2\}}
\Bigl(\prod_{e\in
\overrightarrow{C_i}}a_e^2\Bigr).
\end{align*}
Note that in a given covering there is always an even number of odd cycles,
since otherwise there is no covering of the remaining graph by even cycles. We
now fix a partial covering of the graph by odd cycles, and
sum over coverings of the remaining graph by even cycles. Since the
contribution of the parallel transport to doubled edges cancels out, and since
the matrix $A$ has columns summing to zero, we then
`open' doubled edges according to the complete algorithm, using Remark
\ref{rem:main}. Everything works out in
the same way, with the role of $R$ played by odd cycles. In this case though,
because of the parallel transport, the contributions of RCRSFs do not cancel, but
looking
at the proof of Theorem \ref{thm:main}, we know precisely what those are.
Summing over all partial coverings by odd cycles yields the result.
\end{proof}

\begin{rem}
Theorem \ref{thm:linebundle} can then be specified in the case of bipartite
graphs, in which case there are no odd cycles, in the case of planar graphs or
of graphs embedded on the torus etc. 
\end{rem}

\appendix
\makeatletter
\def\@seccntformat#1{Appendix~\csname the#1\endcsname:\quad}
\makeatother
\section{Pfaffian matrix-tree theorem for 3-graphs and Pfaffian
half-tree theorem for graphs} \label{App:AppendixA}

In the paper \cite{MasbaumVaintrob}, Masbaum and Vaintrob prove a
Pfaffian matrix-tree theorem for spanning trees of 3-uniform hypergraphs. We
start by giving an idea of their result. 

A \emph{3-uniform hypergraph}, or simply
\emph{3-graph} consists of a set of vertices and a set of \emph{hyper-edges},
hyper-edges being triples of vertices. Consider the complete 3-graph
$K^{(3)}_{n+1}$ on the vertex set $\{1,\dots,n+1\}$, where $n$ is even;
hyper-edges consist of the $\binom{n+1}{3}$ possible triples of points. Suppose
that hyper-edges are assigned anti-symmetric weights $y=(y_{ijk})$, that is,
$y_{ijk}=-y_{jik}=y_{jki}$, and $y_{iij}=0$. Note that considering other
3-graphs amounts to setting some of the hyper-edge weights to zero.

A
\emph{spanning tree}
of $K^{(3)}_{n+1}$ is a sub-3-graph spanning all vertices and containing no
cycle; let us denote by $\T^{(3)}$ the set of spanning trees of $K^{(3)}_{n+1}$.
To apprehend spanning trees of 3-graphs, it is helpful to use their bipartite
representation: a hyper-edge is pictured as a Y, where the end points are black
and correspond to vertices of the hyper-edge, and the degree three vertex is
white. Then a sub-3-graph is a spanning tree of $K^{(3)}_{n+1}$ if and only if
its bipartite representation is a spanning tree of the corresponding bipartite
graph, see Figure \ref{fig:Spanning} for an example.

\begin{figure}[ht]
\begin{center}
\includegraphics[width=\linewidth]{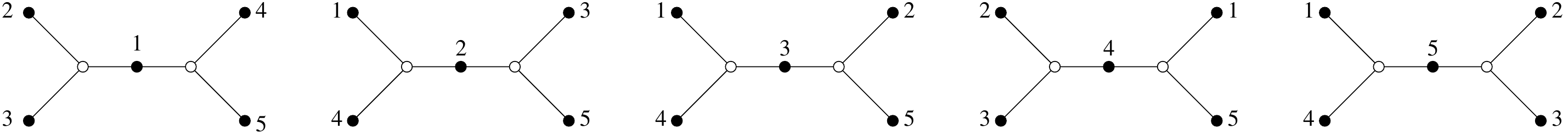}
\caption{Bipartite graph representation of the following 5 spanning trees of
$K^{(3)}_5$: $\{123,145\}$, $\{124,235\}$, $\{134,235\}$, $\{234,145\}$,
$\{145,235\}$. The graph $K^{(3)}_5$ has a total of 15 spanning
trees.}\label{fig:Spanning}
\end{center}
\end{figure}

Define the 
$(n+1)\times(n+1)$ matrix $A^{n+1}=(a_{ij})$ by:
$$
\forall\,i,j\,\in\{1,\dots,n+1\},\quad
a_{ij}=\sum_{k=1}^{n+1}y_{ijk}.
$$
Then, Masbaum and Vaintrob \cite{MasbaumVaintrob} prove that  Pfaffian of the
matrix $A$, obtained from the
matrix $A^{n+1}$ by removing the last line and column, is a signed
$y$-weighted sum over spanning trees of $K^{(3)}_{n+1}$:
\begin{equation}\label{equ:MasbaumVaintrob}
\Pf(A)=\sum_{T\in \T^{(3)}} \sgn(T)\prod_{(i,j,k)\in T}y_{ijk},
\end{equation}
where the product is over all hyper-edges of the spanning tree. We refer to the
original paper \cite{MasbaumVaintrob} for the definition
of $\sgn(T)$. A combinatorial proof of this result is given by Hirschman and
Reiner \cite{HirschmanReiner} and yet another proof using Grassmann variables is
provided by Abdesselam \cite{Abdesselam}. 

Using Sivasubramanian's result \cite{sivasubramanian}, spanning trees of
$K^{(3)}_{n+1}$ can be related to half-spanning trees of the (usual) complete
graph $K_{n+1}$. Sivasubramanian introduces an analog of the Pr\"ufer code for
3-graphs, allowing him to establish a bijection between spanning
trees of $K^{(3)}_{n+1}$ and pairs $(\gamma,M)$, where
$\gamma\in\{1,\dots,n+1\}^{\frac{n}{2}-1}$ and $M$ is a perfect matching of
the (usual) complete graph $K_{n}$ on the vertex set $\{1,\dots,n\}$. This 
bijection is also very clearly explained in the paper \cite{DeMierGoodall} by
Goodall and De Mier. Writing $\M(K_n)$ for the set of perfect matchings of
$K_n$, the set of spanning trees $\T^{(3)}$ can thus
be written as $\cup_{M\in\M(K_n)}\T^{(3)}(M)$, where $\T^{(3)}(M)$ consists of
the spanning trees corresponding to $M$ in the bijection. Equation
\eqref{equ:MasbaumVaintrob} then becomes:
\begin{equation}\label{equ:MasbaumVaintrob2}
\Pf(A)=\sum_{M\in \M(K_n)}\sum_{T\in \T^{(3)}(M)}
\sgn(T)\prod_{(i,j,k)\in T}y_{ijk}.
\end{equation}

\textsc{Example}. When $n+1=5$, spanning trees of $K^{(3)}_5$ are in bijection
with pairs $(\gamma,M)$, where $\gamma\in\{1,\dots,5\}$, and $M$ is a perfect
matching of $K_4$. Returning to the `Pr\"ufer code' of \cite{sivasubramanian},
one sees that the five spanning trees of Figure \ref{fig:Spanning} are in
bijection with the perfect matching $M=\{14,23\}$, and
$\gamma=1,\dots,\gamma=5$, respectively.

We now fix a perfect matching $M$ of $K_n$ and let $T_M$ be one of the
$(n+1)^{\frac{n}{2}-1}$ corresponding spanning trees of $K^{(3)}_{n+1}$. From
$T_M$, we construct a half-spanning tree of $K_{n+1}$ compatible with $M$ as
follows. By the bijection, for every hyper-edge
$ijk$ of $T_M$, exactly one of the pairs $ij,ik,jk$ belongs to $M$; without
loss of generality, let us assume it is $ij$ and that $i<j$. To this hyper-edge,
assign the edge configuration of $K_{n+1}$ consisting of the edge $ij$ and of
the edge $jk$. Repeating this procedure yields a half-tree of $K_{n+1}$
compatible with $M$. It seems that for different $\gamma$'s, the corresponding
half-spanning trees are different.

\textsc{Example}. Recall that Figure \ref{fig:Spanning} consists of the
spanning trees of $K^{(3)}_{5}$ corresponding to the perfect matching
$M=\{14,23\}$ through the `Pr\"ufer code'. Figure \ref{fig:Spanning1} pictures
the half-spanning trees of $K_5$ compatible with $M$ obtained by the above
construction. 

\begin{figure}[ht]
\begin{center}
\includegraphics[width=10cm]{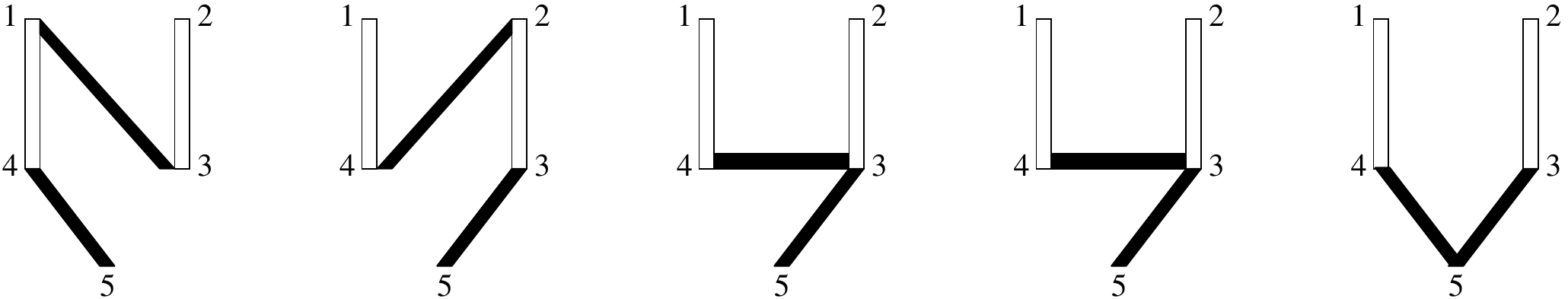}
\caption{Half-spanning trees assigned to spanning
trees of $K^{(3)}_{5}$ of Figure \ref{fig:Spanning}.}\label{fig:Spanning1}
\end{center}
\end{figure}

It is interesting to note that not all half-spanning trees compatible
with $M$ are obtained, and that they do not all satisfy Condition (C) of
Definition \ref{def:def0} (the third one does not satisfy it, see also Figure
\ref{fig:fig10}). A new family of half-spanning trees compatible with $M$ is
constructed; it has $(n+1)^{\frac{n}{2}-1}$ elements, and could probably be
characterized using the `Pr\"ufer code' and the construction of the
half-spanning trees.

This implies that the Pfaffian of the matrix $A$, written using the `Pr\"ufer
code' of \cite{sivasubramanian} as in Equation~\eqref{equ:MasbaumVaintrob2}, can
be expressed as a sum over all perfect matchings $M$ of $K_n$ of a
sum over a new family of half-spanning trees compatible with $M$. 

Now, by the anti-symmetry of the $y$-weights, the matrix $A^{n+1}$ constructed
from the $y$-weights is skew-symmetric and has column sum equal to 0. 
It thus satisfies the hypothesis of Theorem
\ref{thm:main}. Let $M_0$ be a fixed perfect matching of $K_{n}$.
Since the root $R$ consists of a single vertex $n+1$, the
theorem involves half-spanning trees instead of forests, and we denote by
$\T(M_0)$ the set of half-spanning trees compatible with 
$M_0$ of $K_n$, satisfying Condition (C) of Definition~\ref{def:def0}. By
Theorem \ref{thm:main}, we have:
\begin{equation*}\label{equ:thmmain}
\Pf(A)=\sum_{T\in\T(M_0)}\sgn(\sigma_{M_0(T\setminus M_0)})\prod_{e\in
T\setminus M_0}a_e. 
\end{equation*}
Replacing $a_e$ by its definition using $y$-variables, yields
\begin{equation*}
\Pf(A)=\sum_{T\in\T(M_0)}\sgn(\sigma_{M_0(T\setminus M_0)}\prod_{e\in
T\setminus M_0}(\sum_{k=1}^{n+1}y_{ek}).
\end{equation*}
This time, the Pfaffian of $A$ is written as a sum over half-spanning trees
compatible with a \emph{single} fixed perfect matching $M_0$, satisfying
Condition (C).
The term corresponding to a specific half-spanning tree is not a single
\emph{spanning tree} of $K^{(3)}_{n+1}$, but a sum over 3-subgraphs which are
not necessarily trees. To recover the form of \eqref{equ:MasbaumVaintrob}, there
must be cancellations involved.

\textsc{Example}. Take $M_0=\{14,23\}$, and consider the leftmost half-tree
compatible with $M_0$ of Figure \ref{fig:fig10}. Not taking into account signs,
its contribution to $\Pf(A)$ is $a_{42}a_{35}$. Replacing with the $y$-weights,
and using the fact that $y_{iij}=0$ gives a contribution of:
\begin{equation*}
(y_{421}+y_{423}+y_{425})(y_{351}+y_{352}+y_{354})=
y_{421}y_{351}+y_{421}y_{352}+\dots+y_{425}y_{354}.
\end{equation*}
Each term corresponds to a 3-subgraph of $K^{(3)}_5$, but not necessarily a
tree: as soon as a pair of triples of points has more than one index in common,
it is not a tree, for example $y_{425}y_{354}$.

Summarizing, using the `Pr\"ufer code' of \cite{sivasubramanian}, the
Pfaffian matrix-tree theorem of \cite{MasbaumVaintrob} can be written as a sum
over a new family of half-spanning trees, and to each half-spanning tree
corresponds a single spanning tree of $K^{(3)}_{n+1}$. 

When applied to 3-graphs,
our Pfaffian half-tree theorem \ref{thm:main} can be written as a sum over
half-spanning trees  compatible with a
\emph{single} perfect matching $M_0$, satisfying Condition~(C). To each
half-spanning tree corresponds a family of 3-subgraphs of $K^{(3)}_{n+1}$, not
all of which are trees, there are cancellations involved. The Pfaffian
half-tree theorem can be applied in the context of 3-graphs, but the result in
this case is not naturally related to spanning trees of 3-graphs; this theorem
takes its full meaning for (regular) graphs.


\bibliographystyle{alpha}
\bibliography{survey}

\end{document}